\documentclass[1pt]{article}
\usepackage{graphicx}%
\usepackage{multirow}%
\usepackage{amsmath,amssymb,amsfonts}%
\usepackage{amsthm}%
\usepackage{mathrsfs}%
\usepackage[title]{appendix}%
\usepackage{xcolor}%
\usepackage[export]{adjustbox}
\usepackage{bm}
\usepackage{footnotehyper}
\usepackage{amssymb,amsmath,amsthm}
\usepackage{amscd}
\usepackage{hyperref}
\usepackage{extarrows}
\usepackage{lipsum}
\usepackage{epsfig}
\usepackage{amsfonts}
\usepackage{amssymb}
\usepackage{amsmath,enumerate}
\usepackage{commath}
\usepackage{euscript}
\usepackage{enumitem}
\usepackage{amscd}
\usepackage{xcolor}
\usepackage{makecell}
\usepackage{hhline}
\usepackage[numbers,sort&compress]{natbib}
\usepackage{lscape}
\usepackage{textcomp}%
\usepackage{manyfoot}%
\usepackage{booktabs}%
\usepackage{algorithm}%
\usepackage{algorithmicx}%
\usepackage{algpseudocode}%
\usepackage{listings}%
\usepackage{multirow}
\usepackage{amsmath, amsthm, amssymb}
\usepackage{graphicx}
\usepackage{algorithm,setspace}
\usepackage{rotating}
\usepackage{enumitem}
\usepackage{algpseudocode}
  
\usepackage{commath}
\usepackage{booktabs} 

\newtheorem{lemma}{Lemma}  

\usepackage{hyperref}
\hypersetup{
	colorlinks=true,
	linkcolor=blue,
	filecolor=blue,
	urlcolor=blue,
}
\theoremstyle{thmstyleone}%
\newtheorem{theorem}{Theorem}
\theoremstyle{thmstyletwo}%
\newtheorem{remark}{Remark}%

\theoremstyle{thmstylethree}%

\begin{document}
	\title{Parametric bootstrap simultaneous confidence interval  based on doubly type-II censoring}
	\author{Dipak Patra and Lakshmi Kanta Patra
		\footnote{\baselineskip=10pt
			lkpatra@iitbhilai.ac.in;~ patralakshmi@gmail.com}\\
		Department of Mathematics\\
		Indian Institute of Technology Bhilai, Durg, India-491002}
	
	\date{}
	\maketitle
	\begin{abstract}
		Parametric bootstrap and fiducial generalized methods are proposed to construct simultaneous confidence intervals (SCIs) for all pairwise mean differences of several two-parameter exponential distributions, based on doubly censored samples. Both methods have been shown to have correct asymptotic coverage probabilities. Simulation studies are conducted to compare the performance of the two proposed methods. The usefulness of our proposed procedure is illustrated with two examples.\\
	\end{abstract}
	
	
	\section{Introduction} 
	The two-parameter exponential distribution arises naturally in many real-life applications, such as reliability engineering, life-testing experiments, survival analysis, insurance, actuarial science and risk analysis. This distribution is extensively used in modelling lifetime data.  Let $\Pi_i, i=1,2,\dots,k$ be $k$ exponential populations having density  
	\begin{equation}\label{baea}
		f_{X_i}(x;\mu_i,\sigma_i)=\left\{\begin{array}{ll}
			\frac{1}{\sigma_i}\exp\left(-\frac{x-\mu_i}{\sigma_i}\right),\ \, x>\mu_i \\
			0,~~~~~~~~~~~~~~~~~~~~\mbox{Otherwise}.\\
		\end{array}
		\right.
	\end{equation} 
	where $\mu_i \in \mathbb{R}$ is the location parameter, $\sigma_i>0$ is the scale parameter. For each population, define $\theta_{il} = \theta_{i}-\theta_{l}$, where $\theta_i=\mu_i+\sigma_{i},\ \ i\neq l$ and $i , l =1,2,\dots,k$. These pairwise differences of means are the parameter of interest in this study. 
	
	Comparison of the mean lifetimes of different products is an important problem in quality control, reliability analysis, and experimental design for product quality evaluation and maintenance planning. In the last two decades, the problem of finding parametric bootstrap SCIs for the parameters of various distributions has been studied by several authors. \cite{kharrati2013simultaneous} presents simultaneous fiducial generalized confidence intervals for the successive difference of location parameters of a two-parameter exponential distribution, and they also proved that these confidence intervals have correct coverage probability asymptotically. \cite{malekzadeh2014comparing} proposed a parametric bootstrap method to construct SCIs for the differences of location parameters of the treatment groups and control groups. Also, they demonstrate their results through simulation studies and a real data analysis. \cite{sadooghi2014simultaneous} developed a parametric bootstrap approach to propose simultaneous confidence intervals for lognormal mean ratios. Using Monte Carlo simulations, they compared their method with existing GPQ and FGPQ procedures and showed that it achieved more accurate empirical coverage probabilities and substantially shorter confidence intervals.

	\cite{li2015parametric} proposed a parametric bootstrap procedure to propose SCIs for the differences of means of the two-parameter exponential distribution. Also, they proved that the proposed SCIs have correct coverage probability asymptotically and demonstrated this result through simulation studies. Simultaneous pairwise confidence intervals for comparing several inverse Gaussian means under heteroscedasticity were developed by \cite{kharrati2017simultaneous}. The authors used one classic method and two parametric bootstrap methods. Simulation results showed that the bootstrap method achieved coverage probabilities closer to the nominal level. \cite{malekzadeh2020simultaneous} derived SCIs for pairwise quantile difference of several two-parameter exponential distributions under progressive Type-II censoring using classical, generalized fiducial and parametric bootstrap approaches. \cite{thangjai2019simultaneous} proposed SCIs for all pairwise differences of coefficients of variation from several normal populations based on generalized confidence interval, MOVER and computational methods. They also studied the same problem for lognormal distributions (see \cite{thangjai2022simultaneous}).
	\cite{ye2023bootstrap} studied the testing of equality of location parameters in several skew-normal populations with unknown scale and skewness parameters. The authors used a bootstrap test statistic based on moment and maximum-likelihood estimators for thisproblem. 
	\cite{maneerat2021simultaneous} developed SCIs for pairwise mean differences of multiple delta-lognormal populations using parametric bootstrap, FGCI, MOVER and Bayesian credible intervals. \cite{tian2022confidence} studied confidence intervals for the differences of medians of independent lognormal distributions using several approaches, including parametric bootstrap. 
	A bootstrap method for comparing non-normal data with unequal variances was proposed by \cite{johnston2021bootstrap}. \cite{singhasomboon2023asymptotic} developed asymptotic and bootstrap confidence intervals for the ratio of modes of two lognormal distributions.
	
	It is observed that the researchers primarily studied the SCI using an i.i.d. sample. There is little literature available on finding SCI based on a censored sample.    In this work, we study the construction of simultaneous confidence intervals for all pairwise differences of exponential means $\theta_{il}$ under doubly Type-II censoring. In the doubly Type-II censoring scheme, we remove the first $r-1$ lifetimes and the last $n-s$ lifetimes from the set of $n$ ordered lifetimes. 
	
	Let $X_{i1}, X_{i2},\dots, X_{in_i},\ \, i, =1,2,\dots,k$ be a random sample of size $n_{i}$ from $i$-th population. Assume that all observations are mutually independent across and within the samples. 
	For each fixed $i=1,2,\dots,k$,
	$$X_{ir_{i},n_{i}}\leq X_{ir_{i+1},n{i}}\leq \dots \leq X_{is_{i},n_{i}}$$ 
	denote the observed doubly type-II censored sample after removing the smallest $(r_i-1)$ observations and the largest $(n_i-s_i)$  observations, where $r_i<s_i$. We assume  on $r_i-1=n_ip_i$ and $n_i-s_i=n_iq_i$ (See \cite{fernandez2002computing}). Here $p_i$ and $\ q_i $ are fixed censoring proportions satisfying $0\le p_i+q_i < 1$.  We have $(X_{ir_{i},n_i}, V_{i,n_i})$,  is a complete and sufficient statistics for $(\mu_{i}, \sigma_{i})$, where $V_{i,n_{i}}=-(n_i-r_i){X_{ir_{i},n_i}}+{X_{ir_{i+1},n_i}}+\dots+(n_i-s_i+1)
	X_{is_{i},n_i}$. Moreover, $X_{ir_{i},n_i}$ and $V_{i,n_i}$ are independent random variables with probability density functions given by
	$$f_{X_{ir_{i},n_i}}(x)=\frac{n_i!}{\sigma_i(r_i-1)!(n_i-r_i)!}\left[1-e^{-\frac{x-\mu_i}{\sigma_i}}\right]^{r_i-1}\left[e^{-\frac{(x-\mu_i)}{\sigma_i}}\right]^{n_{i}-r_{i}+1},\ \ x>\mu_i,\ \sigma_i>0.$$ and $$g_{V_{i,n_i}}(v)=\frac{1}{\Gamma(s_{i}-r_{i}){\sigma_i}^{s_{i}-r_{i}}}v^{s_{i}-r_{i}-1}e^{-\frac{v}{\sigma_i}},\ \ v>0,\ \sigma_i>0$$
	respectively.  Moreover, we have 
	$$\exp\left(-\frac{X_{ir_{i},n_i}-\mu_{i}}{\sigma_{i}}\right)\sim \mathrm{Beta}(n_{i}-r_{i}+1,r_{i})
	~~~\mbox{and }~~~
	\frac{2V_{i,n_{i}}}{\sigma_i} \sim \chi^2_{2(s_i-r_i)}$$
	
	In this paper, we will consider the SCI's of $\theta_{il}$. The proposed method is based on a parametric bootstrap procedure. The main contribution of this paper is the development of SCIs for the means under doubly Type-II censoring, together with rigorous theoretical justification. In particular, we investigate parametric bootstrap SCIs based on both the unbiased and maximum-likelihood estimators of $\theta_{il}$. For construction of the SCIs, first we uses an unbiased estimator of the variance of the pairwise differences, while the second uses the maximum of the corresponding unbiased variance estimators. The resulting procedures are theoretically justified. Simulation results also shows that the proposed approach perform satisfactorily.

	Rest of the paper is organized as follows. Section \ref{II} develops  the proposed parametric bootstrap SCIs based on unbiased estimators, considering both of the variance estimation schemes. Section \ref{IV} presents simultaneous fiducial generalized confidence interval procedures. Section \ref{V} reports the results of the extensive simulation study assessing the finite-sample performance of the proposed methods, while Section \ref{VI} presents real data applications to demonstrate their practical applicability and effectiveness.
	
	\section{Parametric bootstrap SCIs based on unbiased estimator} \label{II}
	In this section, we develop parametric bootstrap SCIs based on the unbiased estimators for the pairwise differences of the parameter of interest. The SCIs are constructed using two variance estimation schemes: the first uses an unbiased estimator of the variance of the pairwise difference, while the second uses the maximum of the corresponding unbiased variance estimators. The details of the proposed procedures are presented below.
	\subsection{Parametric bootstrap SCIs using the pairwise difference variance}
	In this subsection we will consider the parametric bootstrap SCI of the parameters $\theta_{il}$  based unbiased estimator of $\theta_i$ with an unbiased estimator of the variance of the pairwise difference. 
	After some  simplification we get   
	$$\mathbb{E}(X_{ir_{i},n_i})=\mu_i + \sigma_i \sum_{j=n_i-r_i+1}^{n_i}\frac{1}{j} \  \ \mbox{and} \ \ \mathrm{Var}(X_{ir_{i},n_i})
	=\sigma_i^2\sum_{j=n_i-r_i+1}^{n_i}\frac{1}{j^2}$$ 
	$$\mathbb{E}(V_{i,n_i}) =(s_i-r_i)\sigma_{i} \  \ \mbox{and} \ \  \mathrm{Var}(V_{i,n_{i}})=(s_i-r_i)\sigma_i^2.$$
	So an unbiased estimator of $\theta_i$ is 
	$$\delta_{i} = X_{ir_{i},n_i} + \frac{V_{i,n_i}}{s_i-r_i}\left(1-\sum_{j=n_i-r_i+1}^{n_i}\frac{1}{j}\right), i =1, 2, \dots, k$$ 
	Note that 
	\small{\begin{equation*}
		\mathrm{Var}(\delta_i-\delta_l)
		={\sigma_{i}}^2\sum_{j=n_i-r_i+1}^{n_i}\frac{1}{j^2}+\frac{\sigma_i^2}{s_i-r_i}\left(1-\sum_{j=n_i-r_i+1}^{n_i}\frac{1}{j}\right)^2 +
		{\sigma_{l}}^2\sum_{j=n_l-r_l+1}^{n_l}\frac{1}{j^2}+\frac{\sigma_l^2}{s_l-r_l}\left(1-\sum_{j=n_l-r_l+1}^{n_l}\frac{1}{j}\right)^2,
	\end{equation*}}
	and also we know $$\mathbb{E}\left(\frac{{V^2_{i,n_i}}}{(s_i-r_i)(s_i-r_i+1)}\right) = \sigma_{i}^2.$$
	So an unbiased estimator of $\mathrm{Var}(\delta_i-\delta_l)$ is obtained as,
	
	\begin{align*}
		A_{il}&=\frac{{V^{2}_{i,n_i}}}{(s_i-r_i)(s_i-r_i+1)}\sum_{j=n_i-r_i+1}^{n_i}\frac{1}{j^2}+\frac{{V^{2}_{i,n_i}}}{(s_i-r_i)^2(s_i-r_i+1)}\left(1-\sum_{j=n_i-r_i+1}^{n_i}\frac{1}{j}\right)^2 \\
		&\ \ + \frac{{V^{2}_{l,n_l}}}{(s_l-r_l)(s_l-r_l+1)}\sum_{j=n_l-r_l+1}^{n_l}\frac{1}{j^2}+
		\frac{{V^{2}_{l,n_l}}}{(s_l-r_l)^2(s_l-r_l+1)}\left(1-\sum_{j=n_l-r_l+1}^{n_l}\frac{1}{j}\right)^2.
	\end{align*}
	
	\noindent Define
	\begin{equation} \label{bae1.1}
		T_n:= \max_{i\neq l} \left| \frac{(\delta_i-\delta_l)-(\theta_i - \theta_l)}{\sqrt{A_{il}}}\right|
	\end{equation} 
	
	\noindent The approximate $100(1-\alpha)\%$ two-sided SCI's for $\theta_{i}-\theta_l (i\neq l)$ are 
	$$(\delta_i-\delta_l)\pm q_{\alpha}\sqrt{A_{il}}, i,l = 1, 2, \dots, k(i\neq l),$$ 
	where $q_{\alpha}$ is the approximate $(1-\alpha)$-th quantile of the random variable $T_{n}$.

	Finding exact distribution of $T_{n}$ is very difficult. Therefore we cannot obtain the exact value of $q_{\alpha}$ and hence we cannot construct the SCIs directly. To overcome this difficulty, we use a parametric bootstrap approach based on the distributional property of $T_{n}$. Notice that the distribution of $T_n$ does not depend on the location parameters $\mu_i$'s. Therefore, without loss of generality, we take $\mu_i= 0$ for all $i$. Using these fact, we define the PB version of $T_n$ as
	{\footnotesize\begin{equation} \label{bae1.2}
			T^{\mbox{PB}}_n :=\max_{i \neq l} \left|\frac{ \left(X^{\mbox{PB}}_{ir_{i},n_i} + \frac{V^{\mbox{PB}}_{i,n_i}}{s_i-r_i}\left(1-\sum \limits_{j=n_i-r_i+1}^{n_i}\frac{1}{j}\right)\right)-\left(X^{\mbox{PB}}_{lr_{l},n_l} + \frac{V^{\mbox{PB}}_{l,n_l}}{s_l-r_l}\left(1-\sum\limits_{j=n_l-r_l+1}^{n_l}\frac{1}{j}\right)\right) -(v_i-v_l)}{\sqrt{A^{\mbox{PB}}_{il}}}\right|,
	\end{equation}}
	where $\exp\left(-\frac{X^{\mbox{PB}}_{ir_{i},n_i}}{v_{i}}\right)\sim \mathrm{Beta}(n_{i}-r_{i}+1,r_{i})$,  $\frac{2V^{\mbox{PB}}_{i,n_{i}}}{v_i} \sim \chi^2_{2(s_i-r_i)}$ and $v_i$ is the observed value of $V_i, ~~i=1, 2, \dots, k$ with
	\begin{align*}
		A^{\mbox{PB}}_{il}&=\frac{({V^{\mbox{PB}}_{i,n_i}})^2}{(s_i-r_i)(s_i-r_i+1)}\sum_{j=n_i-r_i+1}^{n_i}\frac{1}{j^2}+\frac{({V^{\mbox{PB}}_{i,n_i}})^2}{(s_i-r_i)^2(s_i-r_i+1)}\left(1-\sum_{j=n_i-r_i+1}^{n_i}\frac{1}{j}\right)^2 \\
		& \ \ + \frac{({V^{\mbox{PB}}_{l,n_l}})^2}{(s_l-r_l)(s_l-r_l+1)}\sum_{j=n_l-r_l+1}^{n_l}\frac{1}{j^2}+
		\frac{({V^{\mbox{PB}}_{l,n_l}})^2}{(s_l-r_l)^2(s_l-r_l+1)}\left(1-\sum_{j=n_l-r_l+1}^{n_l}\frac{1}{j}\right)^2.
	\end{align*}
	Hence, the $100(1-\alpha)\%$ two-sided parametric bootstrap  SCIs of $\theta_i-\theta_l\ \ (i \neq l)$  are given by 
	\begin{equation}\label{2.3}
		(\delta_i-\delta_l) \pm q^{\mbox{PB}}_{\alpha,n} \sqrt{A_{il}}, i,l=1,2,\dots,k(i \neq l)	
	\end{equation}
	where $q^{\mbox{PB}}_{\alpha,n}$ is the $(1-\alpha)$th quantile of the distribution of $T^{\mbox{PB}}_{n}$. The value of $q^{\mbox{PB}}_{\alpha,n}$ is obtained using the following computational procedure.
	
	\noindent \textbf{Computational Procedure:} \label{C1}
	Consider $k$  independent  samples from the corresponding $k$ two-parameter exponential population.
	\begin{itemize}
		\item [(i)] Fix the censoring proportions $p_i$ and $q_i$, and determine $r_i-1=n_ip_i$, $n_i-s_i=n_iq_i$ for each $k$ sample and obtain corresponding $k$ doubly type-ii censored sample.
		\item [(ii) ]From the observed censored data, compute $v_i,$ which is the observed value of $V_i,i=1,2,\dots,k.$
		\item [(iii)] For each sample, generate $X^{\mbox{PB}}_{ir_{i},n_i}$ from $\text{exp}\left(-\frac{X^{\mbox{PB}}_{ir_{i},n_i}}{v_i}\right) \sim
		\text{Beta}(n_i-r_i+1,r_i) $ and $V^{\mbox{PB}}_{i}\sim \frac{v_i}{2}\chi^2_{2(s_i-r_i)}.$ Using these generated bootstrap samples, calculate the bootstrap statistic $T^{\mbox{PB}}_{n}.$
		\item [(iv)] Repeat step (iii) large number of times, say $N,$ to obtain $N$ bootstrap values of $T^{\mbox{PB}}_{n}.$
		\item [(v)] From these $N$ values, obtain its $(1-\alpha)$th quantile as an estimate of $q^{n,PB}_{\alpha}$.
	\end{itemize}
	
	\begin{theorem} \label{T2}
		Let $X_{ir_{i},n_{i}},~ X_{ir_{i+1},n{i}},~ \dots,~ X_{is_{i},n_{i}}, ~i=1,2,\dots,k$ be $k$ independent doubly Type-II censored random samples from the two-parameter exponential distribution $\mathrm{Exp}(\mu_i, \sigma_i)$.  Assume that  $r_{i}-1=n_{i}p_{i}$  and $n_{i}-s_{i}=n_{i}q_{i}$. Suppose that  $\frac{n_i}{n} \to c_i \in (0,1), i=1,2,\dots,k$ as $n \to \infty$, where $n = n_1+n_2+\dots+n_k$. Then 
		$$\mathbb{P}\left(\theta_i-\theta_l \in \left(\delta_i-\delta_l \pm q^{n,PB}_{\alpha}\sqrt{A_{il}}\right) ~~\forall~ i \neq l\right) \to 1- \alpha.$$
	\end{theorem}
	\begin{proof}
		Note that 
		$$\mathbb{P}\left(\theta_i-\theta_l \in \big(\delta_i-\delta_l \pm q^{n,PB}_{\alpha}\sqrt{A_{il}}\big)~~ \forall ~i \neq l\right) = \mathbb{P}(T_n \leq q^{n,PB}_{\alpha} ),$$ 
		where $T_n$ is defined in (\ref{bae1.1}). To establish that the proposed SCIs achieve the desired asymptotic coverage probability, it is sufficient to show $T_{n}$ and  $T^{\mbox{PB}}_n$ have the same limiting distribution as $n \to \infty$, where $T^{\mbox{PB}}_n$ is defined in (\ref{bae1.2}). Define
		\begin{equation}\label{E15}
			T_{n,il}=\frac{(\delta_i-\delta_l)-(\theta_i-\theta_l)}{\sqrt{A_{il}}}, i \neq l, i,l=1,2,\dots,k
		\end{equation}
		and $$T^{\mbox{PB}}_{n,il}=\frac{ \left(X^{\mbox{PB}}_{ir_{i},n_i} + \frac{V^{\mbox{PB}}_{i,n_i}}{s_i-r_i}\left(1-\sum_{j=n_i-r_i+1}^{n_i}\frac{1}{j}\right)\right)-\left(X^{\mbox{PB}}_{lr_{l},n_l} + \frac{V^{\mbox{PB}}_{l,n_l}}{s_l-r_l}\left(1-\sum_{j=n_l-r_l+1}^{n_l}\frac{1}{j}\right)\right) -(v_i-v_l)}{\sqrt{A^{\mbox{PB}}_{il}}}.$$ 
		By the continuous mapping theorem, it suffices to show that
		$T_{n,il}$ and $T^{\mbox{PB}}_{n,il}$ have the same distribution as $n \to \infty$. 
		By Theorem \ref{A1}  we get, as $n_i \to \infty$,  
		\begin{equation}\label{E1}
			\sqrt{n_i}(X_{ir_i,n_i}-\mu_i + \sigma_i \ln(1-p_i))\xlongrightarrow{d} \mathcal{N}\Big(0,\frac{\sigma_i^2 p_i}{1-p_i}\Big)
		\end{equation}
		Also it can be easily seen that
		\begin{equation}\label{E2}
			\sqrt{n_i}\Bigg(\frac{V_{i,n_i}}{s_i-r_i}-\sigma_i\Bigg) \xlongrightarrow{d} \mathcal{N}\Bigg(0,\frac{\sigma_i^2}{1-p_i-q_i}\Bigg)
		\end{equation}
		After some simplification, we have
		\begin{align*}
			\sqrt{n_i}(\delta_{i}-\theta_i)&=\sqrt{n_i}\left(X_{ir_i,n_i}-\mu_{i}+\sigma_{i}\ln(1-p_{i})\right)+\left(1-\sum\limits_{n_i-r_i+1}^{n_i}\frac{1}{j}\right)\sqrt{n_i} \left(\frac{V_{i,n_i}}{s_i-r_i}- \sigma_i\right)\\
			&~~~~~-\sigma_{i}\sqrt{n_i}\left(\ln(1-p_i)+\sum\limits_{n_i-r_i+1}^{n_i}\frac{1}{j}\right)
		\end{align*}
		Using Equations (\ref{E1}), (\ref{E2}) and Lemma \ref{A2}, we obtain
		$$\sqrt{n_i}(\delta_{i}-\theta_i)\xlongrightarrow{d}\mathcal{N}\left(0,\sigma_i^2\left(\frac{p_i}{1-p_i}+\frac{\left(1+\ln(1-p_i)\right)^2}{1-p_i-q_i}\right)\right).$$ 
		Consequently we obtain,  
		$$\sqrt{n}(\delta_{i}-\theta_i)\xlongrightarrow{d}\mathcal{N}\left(0,\frac{\sigma_i^2}{c_i}\left(\frac{p_i}{1-p_i}+\frac{\left(1+\ln(1-p_i)\right)^2}{1-p_i-q_i}\right)\right)$$
		as $\frac{n_i}{n}\to c_i$, where $0<c_i<1$. 
		Similarly, for $\frac{n_l}{n}\to c_l$, where $0<c_l<1$, we get $$\sqrt{n}(\delta_l-\theta_l)\xlongrightarrow{d}\mathcal{N}\left(0,\frac{\sigma_l^2}{c_l}\left(\frac{p_l}{1-p_l}+\frac{\left(1+\ln(1-p_l)\right)^2}{1-p_l-q_l}\right)\right)$$ 
		Since $\delta_{i}$ and $\delta_l$ are independent so we obtain 
		\begin{equation}\label{E5}
			\sqrt{n}\left[(\delta_{i}-\delta_l)-(\theta_i-\theta_l)\right]
			\xlongrightarrow{d}\mathcal{N}\left(0,\frac{\sigma_i^2}{c_i}\left(\frac{p_i}{1-p_i}+\frac{\left(1+\ln(1-p_i)\right)^2}{1-p_i-q_i}\right)+\frac{\sigma_l^2}{c_l}\left(\frac{p_l}{1-p_l}+\frac{\left(1+\ln(1-p_l)\right)^2}{1-p_l-q_l}\right)\right)
		\end{equation}
		Now we will analyse the term $nA_{il}.$ We know that
		\begin{align*}
			nA_{il}&=\frac{{V^{2}_{i,n_i}}}{(s_i-r_i)(s_i-r_i+1)}\sum_{j=n_i-r_i+1}^{n_i}\frac{n}{j^2}+\frac{n{V^{2}_{i,n_i}}}{(s_i-r_i)^2(s_i-r_i+1)}\left(1-\sum_{j=n_i-r_i+1}^{n_i}\frac{1}{j}\right)^2 \\
			&\ \ + \frac{{V^{2}_{l,n_l}}}{(s_l-r_l)(s_l-r_l+1)}\sum_{j=n_l-r_l+1}^{n_l}\frac{n}{j^2}+
			\frac{n{V^{2}_{l,n_l}}}{(s_l-r_l)^2(s_l-r_l+1)}\left(1-\sum_{j=n_l-r_l+1}^{n_l}\frac{1}{j}\right)^2.
		\end{align*}
		A straight forward calculations  shows that
		as $n_i\to\infty$ and $\frac{n_i}{n}\to c_i$, we get 
		\begin{equation}\label{E3}
			\frac{V_{i,n_i}}{s_i-r_i}\xlongrightarrow{P}\sigma_i,~~
			\frac{s_i-r_i}{s_i-r_i+1}\to1,~~
			\sum_{j=n_i-r_i+1}^{n_i}\frac{n_i}{j^2}\to\frac{p_i}{1-p_i}
		\end{equation}
		and 
		\begin{equation}\label{E4}
			\frac{n}{s_i-r_i+1}\to\frac{1}{c_i(1-p_i-q_i)},~~
			\sum_{j=n_i-r_i+1}^{n_i}\frac{n}{j^2}\to\frac{p_i}{c_i(1-p_i)},\quad
			\left(1-\sum_{j=n_i-r_i+1}^{n_i}\frac{1}{j}\right)^2\to\left(1+\ln(1-p_i)\right)^2
		\end{equation}
		Applying the limits in  (\ref{E3}) and (\ref{E4}),  we get
		\begin{equation}\label{E16}
			nA_{il} \xlongrightarrow{P} \frac{\sigma_i^2}{c_i}\left(\frac{p_i}{1-p_i}+\frac{\left(1+\ln(1-p_i)\right)^2}{1-p_i-q_i}\right)+\frac{\sigma_l^2}{c_l}\left(\frac{p_l}{1-p_l}+\frac{\left(1+\ln(1-p_l)\right)^2}{1-p_l-q_l}\right).
		\end{equation}
		and hence 
		\begin{equation}\label{E6}
			\sqrt{nA_{il}} \xlongrightarrow{P} \left[\frac{\sigma_i^2}{c_i}\left(\frac{p_i}{1-p_i}+\frac{\left(1+\ln(1-p_i)\right)^2}{1-p_i-q_i}\right)+\frac{\sigma_l^2}{c_l}\left(\frac{p_l}{1-p_l}+\frac{\left(1+\ln(1-p_l)\right)^2}{1-p_l-q_l}\right)\right]^{\frac{1}{2}}
		\end{equation}  
		Therefore, combining Equations (\ref{E5}), (\ref{E6}) and applying Slutsky's theorem, we conclude that 
		$$T_{n,il}=\frac{(\delta_i-\delta_l)-(\theta_i-\theta_l)}{\sqrt{A_{il}}} \xlongrightarrow{d} \mathcal{N}(0,1)$$ 
		Now we analyse the asymptotic distribution of parametric bootstrap statistics. 
		$$T^{\mbox{PB}}_{n,il}=\frac{ \left(X^{\mbox{PB}}_{ir_{i},n_i} + \frac{V^{\mbox{PB}}_{i,n_i}}{s_i-r_i}\left(1-\sum\limits_{j=n_i-r_i+1}^{n_i}\frac{1}{j}\right)\right)-\left(X^{\mbox{PB}}_{lr_{l},n_l} + \frac{V^{\mbox{PB}}_{l,n_l}}{s_l-r_l}\left(1-\sum\limits_{j=n_l-r_l+1}^{n_l}\frac{1}{j}\right)\right) -(v_i-v_l)}{\sqrt{A^{\mbox{PB}}_{il}}}.$$ 
		We have $X^{\mbox{PB}}_{ir_i,n_i} \overset{d}{=}  v_iZ_{ir_i,n_i}$ where $Z_{ir_i,n_i}$ denotes the $r_{i}$-th order statistic from $\text{Exp}(1)$ sample and $V^{\mbox{PB}}_{i,n_i}\overset{d}{=} \frac{v_i}{2}M_i$ with $M_i \sim \chi^2_{2(s_i-r_i)}$. Hence
		$$T^{\mbox{PB}}_{n,il}\overset{d}{=} \frac{v_i\left(Z_{ir_{i},n_i} + \frac{M_i}{2(s_i-r_i)}\left(1-\sum_{j=n_i-r_i+1}^{n_i}\frac{1}{j}\right)-1\right)-v_l\left(Z_{lr_{l},n_l} + \frac{M_l}{2(s_l-r_l)}\left(1-\sum_{j=n_l-r_l+1}^{n_l}\frac{1}{j}\right)-1\right)}{\sqrt{A^{\mbox{PB}}_{il}}},$$
		where 
		\begin{align*}
			&A^{\mbox{PB}}_{il}\overset{d}{=} \frac{{v_i}^2M^2_i}{4(s_i-r_i)(s_i-r_i+1)}\sum_{j=n_i-r_i+1}^{n_i}\frac{1}{j^2}+\frac{{v_i}^2M^2_i}{4(s_i-r_i)^2(s_i-r_i+1)}\left(1-\sum_{j=n_i-r_i+1}^{n_i}\frac{1}{j}\right)^2 \\
			&+ \frac{{v^2_lM^2_l}}{4(s_l-r_l)(s_l-r_l+1)}\sum_{j=n_l-r_l+1}^{n_l}\frac{1}{j^2}+
			\frac{{v^{2}_{l}}M_l^2}{4(s_l-r_l)^2(s_l-r_l+1)}\left(1-\sum_{j=n_l-r_l+1}^{n_l}\frac{1}{j}\right)^2.
		\end{align*}
		By Theorem \ref{A1}, as $n_i \to \infty$, 
		\begin{equation} \label{E7}
			\sqrt{n_i}(Z_{ir_i,n_i}+ \ln(1-p_i))\xlongrightarrow{d} \mathcal{N}\Big(0,\frac{p_i}{1-p_i}\Big).
		\end{equation}
		Also we have
		\begin{equation}\label{E8}
			\sqrt{n_i}\Big(\frac{M_i}{2(s_i-r_i)}-1\Big) \xlongrightarrow{d} \mathcal{N}\Big(0,\frac{1}{1-p_i-q_i}\Big).	
		\end{equation}
		After some simplification, we get 
		\begin{align*}
			\sqrt{n_i}\left(Z_{ir_{i},n_i} +
			\frac{M_i}{2(s_i-r_i)}\left(1-\sum_{j=n_i-r_i+1}^{n_i}\frac{1}{j}\right)-1\right)&= \sqrt{n_i}\left(Z_{ir_{i},n_i}+\ln(1-p_i)\right)\\
			&+\left(1-\sum\limits_{j=n_i-r_i+1}^{n_i}\frac{1}{j}\right)\sqrt{n_i}\left(\frac{M_i}{2(s_i-r_i)-1}\right)\\
			&-\sqrt{n_i}\left(\ln (1-p_i)+\sum\limits_{j=n_i-r_i+1}^{n_i}\frac{1}{j}\right)
		\end{align*}
		As an application of Equations (\ref{E7}), (\ref{E8}) and Lemma \ref{A2}, we obtain 
		$$\sqrt{n_i}\left(Z_{ir_{i},n_i} +
		\frac{M_i}{2(s_i-r_i)}\left(1-\sum_{j=n_i-r_i+1}^{n_i}\frac{1}{j}\right)-1\right)\xlongrightarrow{d}\mathcal{N}\left(0,\frac{p_i}{1-p_i}+\frac{\left(1+\ln(1-p_i)\right)^2}{1-p_i-q_i}\right)$$ 
		Since $\frac{n_i}{n}\to c_i$, where $0<c_i<1$, it follows that
		\begin{equation}\label{E9}
			\sqrt{n}\left(Z_{ir_{i},n_i} +
			\frac{M_i}{2(s_i-r_i)}\left(1-\sum_{j=n_i-r_i+1}^{n_i}\frac{1}{j}\right)-1\right)\xlongrightarrow{d}\mathcal{N}\left(0,\frac{1}{c_i}\left(\frac{p_i}{1-p_i}+\frac{\left(1+\ln(1-p_i)\right)^2}{1-p_i-q_i}\right)\right)
		\end{equation}
		Similarly we get , 
		\begin{equation}\label{E10}
			\sqrt{n}\left(Z_{lr_{l},n_l} +
			\frac{M_l}{2(s_l-r_l)}\left(1-\sum_{j=n_l-r_l+1}^{n_l}\frac{1}{j}\right)-1\right)\xlongrightarrow{d}\mathcal{N}\left(0,\frac{1}{c_l}\left(\frac{p_l}{1-p_l}+\frac{\left(1+\ln(1-p_l)\right)^2}{1-p_l-q_l}\right)\right)
		\end{equation} 
		As a consequence of independence we get
		\begin{align*}
			&\sqrt{n}\left[\left(Z_{ir_{i},n_i}+\frac{M_i}{2(s_i-r_i)}\left(1-\sum_{j=n_i-r_i+1}^{n_i}\frac{1}{j}\right)-1\right)-\left(Z_{lr_{l},n_l}+\frac{M_l}{2(s_l-r_l)}\left(1-\sum_{j=n_l-r_l+1}^{n_l}\frac{1}{j}\right)-1\right)\right] \\
			&\xlongrightarrow{d}\mathcal{N}\left(0,\frac{1}{c_i}\left(\frac{p_i}{1-p_i}+\frac{\left(1+\ln(1-p_i)\right)^2}{1-p_i-q_i}\right)+\frac{1}{c_l}\left(\frac{p_l}{1-p_l}+\frac{\left(1+\ln(1-p_l)\right)^2}{1-p_l-q_l}\right)\right)
		\end{align*}
		Again we have
		\begin{equation} \label{E11}
			\frac{M_{i}}{2(s_i-r_i)}\xlongrightarrow{P}1
		\end{equation}
		Therefore, using (\ref{E3}), (\ref{E4}), (\ref{E11}) and applying continuous mapping theorem, we have
		\begin{equation} \label{E12}
			\sqrt{nA^{\mbox{PB}}_{il}} \xlongrightarrow{P} \left[\frac{\sigma_i^2}{c_i}\left(\frac{p_i}{1-p_i}+\frac{\left(1+\ln(1-p_i)\right)^2}{1-p_i-q_i}\right)+\frac{\sigma_l^2}{c_l}\left(\frac{p_l}{1-p_l}+\frac{\left(1+\ln(1-p_l)\right)^2}{1-p_l-q_l}\right)\right]^{\frac{1}{2}}
		\end{equation}   
		Consequently as an application of  Slutsky's theorem, we conclude that 
		$$T^{\mbox{PB}}_{n,il} \xlongrightarrow{d} \mathcal{N}(0,1)$$ 
		So both have the same limiting distribution. Note that since the limiting distribution of $T^{\mbox{PB}}_{n,il}$ is continuous, we have $q^{n,PB}_{\alpha} \to q_{\alpha}$, as $n \to \infty$, where $q_\alpha$ is the $(1-\alpha)$-th quantile of the limiting distribution. Therefore as $n \to \infty$, $\mathbb{P}\{T_{n,il} \leq q^{n,PB}_{\alpha}\} \to \mathbb{P}\{Z \leq q_{\alpha}\} = 1-\alpha,$ where $Z$ is the limiting distribution. This completes the proof.
		
	\end{proof}
	
	\subsection{Parametric bootstrap SCIs using maximum of individual variances} \label{III}
	In the previous subsection, the SCIs were constructed using the variance of differences that is $\mathrm{Var}(\delta_i-\delta_l)$. In this section, instead of using the variance of difference,
	we use the maximum of the individual variances, $\max\{\mathrm{Var}(\delta_i),\mathrm{Var}(\delta_l)\}$.
	An unbiased estimator of $\mathrm{Var}(\delta_i)$ is $$A_i=\frac{{V^{2}_{i,n_i}}}{(s_i-r_i)(s_i-r_i+1)}\sum_{j=n_i-r_i+1}^{n_i}\frac{1}{j^2}+
	\left(1-\sum_{j=n_i-r_i+1}^{n_i}\frac{1}{j}\right)^2\frac{{V^{2}_{i,n_i}}}{(s_i-r_i)^2(s_i-r_i+1)}$$
	Define
	\begin{equation} \label{E13}
		T^{*}_n:= \max_{i\neq l} \left| \frac{(\delta_i-\delta_l)-(\theta_i - \theta_l)}{\sqrt{M_{il}}}\right|,
	\end{equation} 
	where $M_{il} = \max\{A_{i},A_{l}\}$. Also the parametric bootstrap version of $T^*_n$ is  {\footnotesize\begin{equation} \label{bae 2.2}
			T^{*^{\mbox{PB}}}_n=\max_{i \neq l} \left|\frac{ \left(X^{\mbox{PB}}_{ir_{i},n_i} + \left(1-\sum_{j=n_i-r_i+1}^{n_i}\frac{1}{j}\right)\frac{V^{\mbox{PB}}_{i,n_i}}{s_i-r_i}\right)-\left(X^{\mbox{PB}}_{lr_{l},n_l} + \left(1-\sum_{j=n_l-r_l+1}^{n_l}\frac{1}{j}\right)\frac{V^{\mbox{PB}}_{l,n_l}}{s_l-r_l}\right) -(v_i-v_l)}{\sqrt{M^{\mbox{PB}}_{il}}}\right|,
	\end{equation}}
	where $\exp\left(-\frac{X^{\mbox{PB}}_{ir_{i},n_i}}{v_{i}}\right)\sim \mathrm{Beta}(n_{i}-r_{i}+1,r_{i})$,  $\frac{2V^{\mbox{PB}}_{i,n_{i}}}{v_i} \sim \chi^2_{2(s_i-r_i)}$ and $v_i$ is the observed value of $V_i, ~~i=1, 2, \dots, k$ with
	\begin{align*}
		M^{\mbox{PB}}_{il}	&=
		\max\Bigg\{\frac{({V^{\mbox{PB}}_{i,n_i}})^2}{(s_i-r_i)(s_i-r_i+1)}\sum_{j=n_i-r_i+1}^{n_i}\frac{1}{j^2}+\left(1-\sum_{j=n_i-r_i+1}^{n_i}\frac{1}{j}\right)^2\frac{({V^{\mbox{PB}}_{i,n_i}})^2}{(s_i-r_i)^2(s_i-r_i+1)},\\
		&\qquad\qquad\frac{({V^{\mbox{PB}}_{l,n_l}})^2}{(s_l-r_l)(s_l-r_l+1)}\sum_{j=n_l-r_l+1}^{n_l}\frac{1}{j^2}	+\left(1-\sum_{j=n_l-r_l+1}^{n_l}\frac{1}{j}\right)^2\frac{({V^{\mbox{PB}}_{l,n_l}})^2}
		{(s_l-r_l)^2(s_l-r_l+1)}
		\Bigg\}.
	\end{align*}

	Hence, the $100(1-\alpha)\%$ two-sided parametric bootstrap  SCIs of $(\theta_i-\theta_l),~ i \neq l$  are given by 
	\begin{equation}\label{3.3}
		(\delta_i-\delta_l) \pm q^{\mbox{PB}}_{\alpha,n} \sqrt{M_{il}}, i,l=1,2,\dots,k(i \neq l)
	\end{equation} 
	where $q^{\mbox{PB}}_{\alpha,n}$ is the $(1-\alpha)$th quantile of the distribution of $T^{*^{\mbox{PB}}}_n$. To obtain $q^{\mbox{PB}}_{\alpha,n}$, we use the same type of computational algorithm as described in the previous section. For completeness, the computational procedure is given below.
	
	\noindent \textbf{Computational procedure:} \label{C2}
	Consider $k$  independent  samples from the corresponding $k$ two-parameter exponential population.
	\begin{itemize}
		\item [(i)] Fix the censoring proportions $p_i$ and $q_i$, and determine $r_i-1=n_ip_i$, $n_i-s_i=n_iq_i$ for each $k$ sample and obtain corresponding $k$ doubly type-ii censored sample.
		\item [(ii) ]From the observed censored data, compute $v_i,$ which is the observed value of $V_i,i=1,2,\dots,k.$
		\item [(iii)] For each sample, generate $X^{\mbox{PB}}_{ir_{i},n_i}$ from $\text{exp}\left(-\frac{X^{\mbox{PB}}_{ir_{i},n_i}}{v_i}\right) \sim
		\text{Beta}(n_i-r_i+1,r_i) $ and $V^{\mbox{PB}}_{i}\sim \frac{v_i}{2}\chi^2_{2(s_i-r_i)}.$ Using these generated bootstrap samples, calculate the bootstrap statistic $T^{*^{\mbox{PB}}}_n.$
		\item [(iv)] Repeat step (iii) large number of times, say $N,$ to obtain $N$ bootstrap values of $T^{*^{\mbox{PB}}}_n.$
		\item [(v)] From these $N$ values, obtain its $(1-\alpha)$th quantile as an estimate of $q^{n,PB}_{\alpha}$.
	\end{itemize}

	\begin{theorem}\label{T3}
		Let $X_{ir_{i},n_{i}},~ X_{ir_{i+1},n{i}},~ \dots,~ X_{is_{i},n_{i}}, ~i=1,2,\dots,k$ be $k$ independent doubly Type-II censored samples from the two-parameter exponential distribution $\mathrm{Exp}(\mu_i, \sigma_i)$.  Assume that  $r_{i}-1=n_{i}p_{i}$  and $n_{i}-s_{i}=n_{i}q_{i}$. Suppose that  $\frac{n_i}{n} \to c_i \in (0,1), i=1,2,\dots,k$ as $n \to \infty$, where $n = n_1+n_2+\dots+n_k$. Then 
		$$\mathbb{P}\left(\theta_i-\theta_l \in \left(\delta_i-\delta_l \pm q^{n,PB}_{\alpha}\sqrt{M_{il}}\right) ~~\forall~ i \neq l\right) \to 1- \alpha.$$
	\end{theorem}
	
	\noindent \textbf{Proof:}
	Note that $$\mathbb{P}\left(\theta_i-\theta_l \in \big(\delta_i-\delta_l \pm q^{n,PB}_{\alpha}\sqrt{M_{il}}\big)~~ \forall ~i \neq l\right) = \mathbb{P}(T_n \leq q^{n,PB}_{\alpha} ),$$  where $T^*_n$ is defined in (\ref{E13}). To establish that the proposed SCIs achieve the desired asymptotic coverage probability, it is sufficient to show $T^*_{n}$ and  $T^{*^{\mbox{PB}}}_n$ have the same limiting distribution as $n \to \infty$, where $T^{*^{\mbox{PB}}}_n$ is defined in (\ref{bae 2.2}). By the continuous mapping theorem, it further suffices to show that
	$T^*_{n,il}$ and $T^{*^{\mbox{PB}}}_{n,il}$ have the same distribution as $n \to \infty$, where  $T^*_{n,il}$ and $T^{*^{\mbox{PB}}}_{n,il}$ are defined as
	$$T^{*}_{n,il}=\frac{(\delta_i-\delta_l)-(\theta_i-\theta_l)}{\sqrt{M_{il}}}, i \neq l, i,l=1,2,\dots,k$$  and $$T^{*^{\mbox{PB}}}_{n,il}=\frac{ \left(X^{\mbox{PB}}_{ir_{i},n_i} + \left(1-\sum_{j=n_i-r_i+1}^{n_i}\frac{1}{j}\right)\frac{V^{\mbox{PB}}_{i,n_i}}{s_i-r_i}\right)-\left(X^{\mbox{PB}}_{lr_{l},n_l} + \left(1-\sum_{j=n_l-r_l+1}^{n_l}\frac{1}{j}\right)\frac{V^{\mbox{PB}}_{l,n_l}}{s_l-r_l}\right) -(v_i-v_l)}{\sqrt{M^{\mbox{PB}}_{il}}}.$$ 
	From the proof of the previous result, we have 
	\begin{equation}\label{M1}
		\sqrt{n}\left[(\delta_{i}-\delta_l)-(\theta_i-\theta_l)\right]
		\xlongrightarrow{d}\mathcal{N}\left(0,\frac{\sigma_i^2}{c_i}\left(\frac{p_i}{1-p_i}+\frac{\left(1+\ln(1-p_i)\right)^2}{1-p_i-q_i}\right)+\frac{\sigma_l^2}{c_l}\left(\frac{p_l}{1-p_l}+\frac{\left(1+\ln(1-p_l)\right)^2}{1-p_l-q_l}\right)\right)	
	\end{equation}
	and 
	\begin{equation}\label{M2}
		nM_{il} \xlongrightarrow{P} \max \left\{\frac{\sigma_i^2}{c_i}\left(\frac{p_i}{1-p_i}+\frac{\left(1+\ln(1-p_i)\right)^2}{1-p_i-q_i}\right), \frac{\sigma_l^2}{c_l}\left(\frac{p_l}{1-p_l}+\frac{\left(1+\ln(1-p_l)\right)^2}{1-p_l-q_l}\right)\right\}	
	\end{equation}
	Therefore, combining Equations (\ref{M1}), (\ref{M2}) and applying Slutsky's theorem, we conclude that

	$$T^{*}_{n,il}\xlongrightarrow{d} \mathcal{N}\left(0, \frac{\frac{\sigma_i^2}{c_i}\left(\frac{p_i}{1-p_i}+\frac{\left(1+\ln(1-p_i)\right)^2}{1-p_i-q_i}\right)+\frac{\sigma_l^2}{c_l}\left(\frac{p_l}{1-p_l}+\frac{\left(1+\ln(1-p_l)\right)^2}{1-p_l-q_l}\right)}{\sqrt{\max\Bigg\{\frac{\sigma_i^2}{c_i}\left(\frac{p_i}{1-p_i}+\frac{\left(1+\ln(1-p_i)\right)^2}{1-p_i-q_i}\right),\frac{\sigma_l^2}{c_l}\left(\frac{p_l}{1-p_l}+\frac{\left(1+\ln(1-p_l)\right)^2}{1-p_l-q_l}\right)\Bigg\}}}\right) $$
	
	\noindent Similarly we can prove that
	$$T^{*^{\mbox{PB}}}_{n,il}\xlongrightarrow{d} \mathcal{N}\left(0, \frac{\frac{\sigma_i^2}{c_i}\left(\frac{p_i}{1-p_i}+\frac{\left(1+\ln(1-p_i)\right)^2}{1-p_i-q_i}\right)+\frac{\sigma_l^2}{c_l}\left(\frac{p_l}{1-p_l}+\frac{\left(1+\ln(1-p_l)\right)^2}{1-p_l-q_l}\right)}{\sqrt{\max\Bigg\{\frac{\sigma_i^2}{c_i}\left(\frac{p_i}{1-p_i}+\frac{\left(1+\ln(1-p_i)\right)^2}{1-p_i-q_i}\right),\frac{\sigma_l^2}{c_l}\left(\frac{p_l}{1-p_l}+\frac{\left(1+\ln(1-p_l)\right)^2}{1-p_l-q_l}\right)\Bigg\}}}\right) $$
	So, both have the same limiting distribution. Hence the theorem is proved.
	
	\section{Simultaneous fiducial generalized confidence intervals}\label{IV}
	In this section, we develop simultaneous fiducial generalized confidence intervals (FGCIs) for pairwise differences of means 
	$$ \theta_{il} = \theta_{i} -\theta_{l},~~ i,l=1,2,\dots,k,\ \ i \neq l. $$
	We will consider the fiducial generalized pivotal quantity (FGPQ) approach of  \cite{hannig2006fiducial}. 
	Let $X^*_{ir_i,n_i}$, $V^*_i$ denote independent copies of $X_{ir_i,n_i}$, $V_i$, respectively. It is well known that  $X^*_{ir_i,n_i}$ and $V^*_i$  are independently distributed with $$X^*_{ir_i,n_i} \overset{d}{=} \mu_i + \sigma_{i}Z^*_i$$ where $Z^*_i$ is the $r_i$-th order statistic of $\mathrm{Exp}(1)$ random variables and $$V^*_{i,n_i} \overset{d}{=}  \frac{\sigma_i}{2}W_i, ~~W_i \sim \chi^2_{2(s_i-r_i)}$$
	Cosider the FGPQs of $\mu_i$ and $\sigma_{i}$, as
	$$R_{\mu_i}=X_{ir_i,n_i}-(X^*_{ir_i,n_i}-\mu_i)\frac{V_{i,n_i}}{V^*_{i,n_i}} ~~,~~ R_{\sigma_i}=\frac{V_{i,n_i}}{V^*_{i,n_i}}\sigma_i,~ i=1,2,\dots,k$$
	Consequently, for $\theta_{i}=\mu_{i}+\sigma_{i},$ the  FGPQ is 
	$$R_{\theta_{i}}(X,X^*,\delta)=X_{ir_i,n_i}-(X^*_{ir_i,n_i}-\mu_i-\sigma_i)\frac{V_{i,n_i}}{V^*_{i,n_i}}$$ 
	Thus, the FGPQ for the pairwise difference $\theta_{il}=\theta_{i}-\theta_{l}$ is defined as 
	$$R_{\theta_{il}}(X,X^*,\delta)=R_{\theta_{i}}(X,X^*,\delta)-R_{\theta_{l}}(X,X^*,\delta).$$ 
	To construct simultaneous confidence intervals, define
	\begin{equation} \label{F1}
		F_n= \max_{i\neq l} \left|\frac{(\delta_i-\delta_l)-R_{\theta_{il}}
			(X,X^*,\delta)}{\sqrt{F_{il}}}\right|,
	\end{equation}  
	where $F_{il}$ is an unbiased estimator of the variance of $\delta_i-\delta_l$ is given as
	\begin{align*}
		F_{il}&=\frac{{V^{2}_{i,n_i}}}{(s_i-r_i)(s_i-r_i+1)}\sum_{j=n_i-r_i+1}^{n_i}\frac{1}{j^2}+\left(1-\sum_{j=n_i-r_i+1}^{n_i}\frac{1}{j}\right)^2\frac{{V^{2}_{i,n_i}}}{(s_i-r_i)^2(s_i-r_i+1)} \\
		&\ \ + \frac{{V^{2}_{l,n_l}}}{(s_l-r_l)(s_l-r_l+1)}\sum_{j=n_l-r_l+1}^{n_l}\frac{1}{j^2}+
		\left(1-\sum_{j=n_l-r_l+1}^{n_l}\frac{1}{j}\right)^2\frac{{V^{2}_{l,n_l}}}{(s_l-r_l)^2(s_l-r_l+1)}.
	\end{align*}
	
	Hence, the $100(1-\alpha)\%$ two-sided FGCIs of $\theta_i-\theta_l\ \ (i \neq l)$  are given by 
	$$(\delta_i-\delta_l) \pm f_{\alpha,n} \sqrt{A_{il}}, i,l=1,2,\dots,k(i \neq l)$$ 
	where $f_{\alpha,n}$ is the $(1-\alpha)$th quantile of the distribution of $F_{n}$. The value of $f_{\alpha,n}$ is obtained using the following computational procedure.
	\vspace{0.25cm}

	\noindent \textbf{Computational Procedure:} Consider $k$  independent  samples from the corresponding $k$ two-parameter exponential population.
	\begin{itemize}
		\item [(i)] Fix the censoring proportions $p_i$ and $q_i$, and determine $r_i-1=n_ip_i$, $n_i-s_i=n_iq_i$ for each $k$ sample and obtain corresponding $k$ doubly type-ii censored sample.
		\item [(ii) ]From the observed censored data, compute the observed values $x_{ir_i,n_i}$ and $v_i,$ of $X_{ir_i,n_i}, V_i,i=1,2,\dots,k$ respectively and calculate the corresponding estimator $\delta_{i},i=1,2,\dots,k $.
		\item [(iii)] Generate a realizations of $R_{\theta_{il}}$ and then calculate the value of $F_n$ defined in (\ref{F1}).
		\item [(iv)] Repeat Step (iii) independently $N$ times. The empirical $(1-\alpha)$th quantile of the resulting $N$ values of $F_n$ is taken as the estimate of $f_{\alpha, n}$.
	\end{itemize}
	\begin{theorem}
		Let $X_{ir_{i},n_{i}},~ X_{ir_{i+1},n{i}},~ \dots,~ X_{is_{i},n_{i}}, ~i=1,2,\dots,k$ be $k$ independent doubly Type-II censored random samples from the two-parameter exponential distribution $\mathrm{Exp}(\mu_i, \sigma_i)$.  Assume that  $r_{i}-1=n_{i}p_{i}$  and $n_{i}-s_{i}=n_{i}q_{i}$. Suppose that  $\frac{n_i}{n} \to c_i \in (0,1), i=1,2,\dots,k$ as $n \to \infty$, where $n = n_1+n_2+\dots+n_k$. Then 
		$$\mathbb{P}\left(\theta_i-\theta_l \in \left(\delta_i-\delta_l \pm f_{\alpha, n}\sqrt{F_{il}}\right) ~~\forall~ i \neq l\right) \to 1- \alpha.$$
	\end{theorem}
	\begin{proof}
		Note that 
		$$\mathbb{P}\left(\theta_i-\theta_l \in \big(\delta_i-\delta_l \pm f_{\alpha, n}\sqrt{F_{il}}\big)~~ \forall ~i \neq l\right) = \mathbb{P}(T_n \leq f_{\alpha, n}),$$
		where $T_n$ is defined in (\ref{bae1.1}). To establish that the proposed SCIs achieve the desired asymptotic coverage probability, it is sufficient to show $T_{n}$ and  $F_n$ have the same limiting distribution as $n \to \infty$, where $F_n$ is defined in (\ref{F1}). By the continuous mapping theorem, it further suffices to show that
		$T_{n,il}$ and $F_{n,il}$ have the same distribution as $n \to \infty$, where  $T_{n,il}$ and $F_{n,il}$ are defined as
		$$T_{n,il}=\frac{(\delta_i-\delta_l)-(\theta_i-\theta_l)}{\sqrt{A_{il}}}, i \neq l, i,l=1,2,\dots,k$$  and $$F_{n,il}=\frac{(\delta_{i}-\delta_{l})-R_{\theta_{il}}
			(X,X^*,\delta)}{\sqrt{A_{il}}}.$$ 

		As established in the proof of Theorem \ref{T2}, we have
		$$T_{n,il}=\frac{(\delta_i-\delta_l)-(\theta_i-\theta_l)}{\sqrt{A_{il}}} \xlongrightarrow{d} \mathcal{N}(0,1)$$ 
		Now we analyse the asymptotic distribution of FGPQ. Recall that 
		$$F_{n,il}=\frac{(\delta_{i}-\delta_{l})-R_{\theta_{il}}(X,X^*,\delta)}{\sqrt{A_{il}}}.$$
		We can write $$F_{n,il}=F_{n,i}-F_{n,l}$$. 
		
		First, consider the numerator of $F_{n,i}$. It can be expressed as 
		{\small\begin{align*}
			&\frac{V_{i,n_i}}{V^*_{i,n_i}}\left(\left(X^*_{i,r{i},n_{i}}-\mu_{i}-\sigma_{i}\right)+\frac{V^*_{i,n_{i}}}{s_i-r_i}\left(1-\sum_{n_i-r_i+1}^{n_i}\frac{1}{j}\right)\right)\\
			&= \frac{V_{i,n_i}}{V^*_{i,n_i}}\left(\left(X^*_{i,r{i},n_{i}}-\mu_{i}-\sigma_{i}\ln(1-p_i)\right)+\left(\frac{V^*_{i,n_{i}}}{s_i-r_i}-\sigma_{i}\right)\left(1-\sum_{n_i-r_i+1}^{n_i}\frac{1}{j}\right)-\sigma_{i}\left(\ln(1-p_i)+\sum_{j=n_i-r_i+1}^{n_i}\frac{1}{j}\right)\right).
		\end{align*}}
		Since $X^*_{ir_i,n_i}$ and $V^*_{i,n_i}$ are independent copy of $X_{ir_i,n_i}$ and $V_{i,n_i}$ respectively and have the same distributions, the asymptotic results obtained previously can be applied directly. Hence $$\sqrt{n}\left(\left(X^*_{i,r{i},n_{i}}-\mu_{i}-\sigma_{i}\right)+\frac{V^*_{i,n_{i}}}{s_i-r_i}\left(1-\sum_{n_i-r_i+1}^{n_i}\frac{1}{j}\right)\right) \xlongrightarrow{d}\mathcal{N}\left(0,\frac{\sigma_i^2}{c_i}\left(\frac{p_i}{1-p_i}+\frac{\left(1+\ln(1-p_i)\right)^2}{1-p_i-q_i}\right)\right).$$ Similarly, $$\sqrt{n}\left(\left(X^*_{l,r{l},n_{l}}-\mu_{l}-\sigma_{l}\right)+\frac{V^*_{l,n_{l}}}{s_l-r_l}\left(1-\sum_{n_l-r_l+1}^{n_l}\frac{1}{j}\right)\right) \xlongrightarrow{d}\mathcal{N}\left(0,\frac{\sigma_l^2}{c_l}\left(\frac{p_l}{1-p_l}+\frac{\left(1+\ln(1-p_l)\right)^2}{1-p_l-q_l}\right)\right).$$
		Also, 
		$$\frac{V_{i,n_i}}{V^*_{i,n_i}} \xlongrightarrow{P} 1.$$
		Therefore, by slutky's theorem, the numerator of $F_{n,il}$ satisfies
		$$\sqrt{n}\left(\mbox{numerator of} F_{n,il}\right) \xlongrightarrow{d}\mathcal{N}\left(0,\frac{\sigma_i^2}{c_i}\left(\frac{p_i}{1-p_i}+\frac{\left(1+\ln(1-p_i)\right)^2}{1-p_i-q_i}\right)+\frac{\sigma_l^2}{c_l}\left(\frac{p_l}{1-p_l}+\frac{\left(1+\ln(1-p_l)\right)^2}{1-p_l-q_l}\right)\right)$$
		Furthermore, from Equation (\ref{E16}), we get $$\text{denominator of} F_{n,i} =nA_{il} \xlongrightarrow{P} \frac{\sigma_i^2}{c_i}\left(\frac{p_i}{1-p_i}+\frac{\left(1+\ln(1-p_i)\right)^2}{1-p_i-q_i}\right)+\frac{\sigma_l^2}{c_l}\left(\frac{p_l}{1-p_l}+\frac{\left(1+\ln(1-p_l)\right)^2}{1-p_l-q_l}\right).$$ 
		Consequently, applying Slutsky's theorem, we obtain 
		$$F_{n,il} \xlongrightarrow{d} \mathcal{N}(0,1)$$ 
		So both have the same limiting distribution. Note that since the limiting distribution of $T^{\mbox{PB}}_{n,il}$ is continuous, we have $q^{n,PB}_{\alpha} \to q_{\alpha}$, as $n \to \infty$, where $q_\alpha$ is the $(1-\alpha)$-th quantile of the limiting distribution. Therefore as $n \to \infty$, $\mathbb{P}\{T_{n,il} \leq q^{n,PB}_{\alpha}\} \to \mathbb{P}\{Z \leq q_{\alpha}\} = 1-\alpha,$ where $Z$ is the limiting distribution. This completes the proof.
		
	\end{proof}
	
	\begin{remark}
		In Section \ref{IV}, all the results are derived using the unbiased estimator of $\theta_{il}$ together with an unbiased estimator of $\mbox{Var}(\theta_i-\theta_l)$. Alternatively, instead of using unbiased estimator of $\mbox{Var}(\theta_i-\theta_l)$, one may use the maximum of unbiased estimators of $\mbox{Var}(\theta_i)$ and $\mbox{Var}(\theta_l)$. The resulting procedure leads to the same as those obtained in Section \ref{III} for the parametric bootstrap method. In the present setting, the corresponding arguments apply to the fiducial method as well. The required results can be established using arguments similar to those given in the previous section. 
	\end{remark}
	\begin{remark}
		In Sections \ref{II} and \ref{IV}, the results were established using the unbiased estimator of $\theta_{il}$. The unbiased estimator was used in the corresponding expressions and derivations in both sections. The same results also hold when the maximum likelihood estimator (MLE) of $\theta_{il}$ is used. Indeed, by replacing the unbiased estimator with the MLE in the corresponding expressions, the results follow directly using the same arguments and theorems as above. Thus, all the results of the previous sections remain valid for the MLE as well.	
	\end{remark}
	
	\begin{remark}
		In many experimental studies, the subjects are divided into a control group and one or more treatment groups. The control group represents the existing condition, while the treatment group receives new conditions. The primary objective is often to determine whether the treatments produce a meaningful changed compared with the control group. In statistical analysis, this is commonly studied through the differences between the treatment and control group means.
		
		The proposed method can be directly extended to this setting. Let $i=1,2,\dots,k_1$ denote the treatment groups and $l=1,2,\dots,k_2$ denote the control groups, with corresponding population means $\theta_i$ and $\theta_l$, respectively. Our problem can therefore be viewed as a special case of treatment-control comparisons, where the parameter of interest are the treatment-control mean differences 
		$$\theta_{il}=\theta_{i}-\theta_{l},~~i=1,2,\dots,k_1,~~l=1,2,\dots,k_2$$
		Thus, all treatment-control mean differences are considered simultaneously.
		
		The parametric bootstrap procedures proposed in Section \ref{II} and fiducial procedure proposed in Section \ref{IV} can be used to construct simultaneous confidence intervals for $\theta_{il}$. The construction is essentially the same as that described therein, with the pairwise comparisons restricted to treatment-control pairs. Thus, the proposed method can be directly extended to simultaneous inference for treatment-control studies.
	\end{remark}	
	
	\begin{remark}
		The proposed methods can be naturally extended to construct
		SCIs for the pairwise differences between the location parameters of the $k$ populations. In the present setting, let $\mu_i$ denote the location parameter corresponding to the $i$-th population, for $i=1,2,\dots,k$. Therefore our parameter of interest is now the pairwise difference 
		$$\gamma_{il}= \mu_{i}-\mu_{l},~~i,l=1,2,\dots,k,~i \neq l.$$
		Thus, the objective is to construct SCIs for all pairwise differences 
		$\gamma_{il},~~i,l=1,2,\dots,k,~~i\neq l$, simultaneously. For the PB method, we use the same estimator of $\mu_{i}$ and same type of statistic as considered in the previous section. That is, no new estimator needs to be introduced the estimator of $\mu_i$, along with its associated variance estimator continues to play the same role here as it did in the case of the mean parameter. The corresponding statistic for the pairwise difference is defined as 
		\begin{equation} 
			T :=\max_{i \neq l} \left|\frac{ \left(X_{ir_{i},n_i} + \frac{V_{i,n_i}}{s_i-r_i}\sum \limits_{j=n_i-r_i+1}^{n_i}\frac{1}{j}\right)-\left(X_{lr_{l},n_l} + \frac{V_{l,n_l}}{s_l-r_l}\sum\limits_{j=n_l-r_l+1}^{n_l}\frac{1}{j}\right) -(\mu_i-\mu_l)}{\sqrt{A_{il}}}\right|,
		\end{equation}
		where $A_{il}$ is an unbiased estimator of $\mbox{Var}(\mu_i-\mu_l)$.  Therefore, the parametric bootstrap version of the $T$ is $T^{PB}$ which is defined by replacing each quantity in $T$ with its parametric bootstrap counterpart and setting $\mu_i, \mu_j$ to be zero as $T$ is independent of them as 
		\begin{equation} \label{bae 2.2}
			T^{^{\mbox{PB}}}=\max_{i \neq l} \left|\frac{ \left(X^{\mbox{PB}}_{ir_{i},n_i} +\frac{V^{\mbox{PB}}_{i,n_i}}{s_i-r_i}\sum\limits_{j=n_i-r_i+1}^{n_i}\frac{1}{j}\right)-\left(X^{\mbox{PB}}_{lr_{l},n_l} +\frac{V^{\mbox{PB}}_{l,n_l}}{s_l-r_l}\sum\limits_{j=n_l-r_l+1}^{n_l}\frac{1}{j}\right)}{\sqrt{A^{\mbox{PB}}_{il}}}\right|.
		\end{equation}
		
		The construction of the SCIs then proceeds in exactly the same manner as described in the previous section. The quantiles of the bootstrap distribution of $T^{PB}$ are obtained through Monte Carlo simulation and these quantiles are used in place of the unknown quantiles of $T$ to build the confidence intervals for each $\gamma_{il}$. In particular, the same theoretical results and the same parametric bootstrap arguments employed in the previous section can be used here  and all the results proved therein continue to hold in an identical manner for all pairwise differences of $\gamma_{il}$.
		
		This problem can also addressed using the fiducial method. For the fiducial method, all estimators are the same as those used in the PB method. Therefore, it suffices to define the corresponding fiducial statistic as
		\begin{equation} \label{F1}
			F_n= \max_{i\neq l} \left|\frac{\mu_i-\mu_l-R_{\mu_{il}}
				(X,X^*,\delta)}{\sqrt{F_{il}}}\right|,
		\end{equation} 
		where $F_{il}$ is an unbiased estimator of the variance of $(\mu_i-\mu_l)$, $$R_{\mu_i}(X,X^*,\delta)=X_{ir_i,n_i}-(X^*_{ir_i,n_i}-\mu_i)\frac{V_{i,n_i}}{V^*_{i,n_i}},$$
		and 
		$$R_{\mu_{il}}(X,X^*,\delta) =R_{\mu_l}(X,X^*,\delta)-R_{\mu_i}(X,X^*,\delta) ~ i,l=1,2,\dots,k,~~i\neq l.$$
		The procedure for constructing the simultaneous confidence intervals (SCIs) for $\gamma_{il}$ is the same as that described is Section \ref{IV}.
	\end{remark}
	
\section{Simulation Studies}\label{V}
In this section, we conduct simulation studies to assess the performance of the proposed PB SCIs for differences of means of several two-parameter exponential distributions based on censored samples. For comparison, the FG SCIs  are also considered. The performance of both methods is evaluated in terms of their empirical coverage probabilities and average volumes of the resulting confidence intervals. Specifically, the empirical coverage probabilities of the PB SCIs are obtained using the following simulation algorithm:

\vspace{0.25cm}

\noindent \textbf{Algorithm:}
\begin{itemize}
	\item [(i)]  Generate random samples $X_{i1}, \dots, X_{in_i}$ from $\mathrm{Exp}(\mu_i, \sigma_i)$, for $i=1,\dots,k$.
	\item [(ii)] Calculate the critical value
	$q^{n,PB}_{\alpha}$, using computational procedure described in Section \ref{II} with $N=10,000$ bootstrap replications.
	\item [(iii)] Using the samples generated in step (i) and the critical value $q^{n,PB}_{\alpha}$ obtained in step (ii), construct the SCIs according to (\ref{2.3}) and (\ref{3.3}). Record whether all the differences $\theta_i-\theta_l(i \neq l)$, are simultaneously contained in their corresponding confidence intervals.
	\item [(iv)]  Repeat steps (i)--(iii), $B=10,000$ times. The empirical coverage probability
	is then estimated by the proportion of simulation runs for which all the differences
	$\theta_i-\theta_l(i \neq l)$, are simultaneously contained in
	their corresponding SCIs.
\end{itemize}

The empirical coverage probabilities for the FG SCIs are obtained using a similar simulation procedures except for the different resampling schemes. For each procedure, the average volume (AV) of the SCIs is calculated as the mean of the products of the lengths of all individual confidence intervals. When $k=2$, the AV is reduces to the simply average lengths of the confidence intervals.

In the simulation, five configuration factors are considered to evaluate the performance of the two SCI procedures: the number of two parameter exponential populations, the sample sizes, the censoring proportion, the values of $\mu_i$ and $\theta_i$, $i=1,\dots,k$. The specific configurations considered for these factors are given inthe following table. A nominal confidence level of $95\%$ is used throughout the study. The simulation results for $k=2,3,4$ are reported in the following tables. For simplicity, CP is denotes the coverage probability in all three cases.

\noindent \textbf{Simulation results I:}
We first consider two cases with $k=2$ groups, under two different parameters settings, five censoring proportions and six different samples size configurations. It is observed from Tables \ref{table1}, \ref{table2}, \ref{table3} and \ref{table4} that, the coverage probabilities of both the PB and fiducial methods are close to the nominal level. The resulting confidence intervals are essentially the same for both methods, as evidence by their nearly identical AV. For both approaches, the choice of estimators has little effect on the result and the choice between the estimated sum and maximum of variances has a negligible effect for the PB method and no effect for the fiducial method. Under heavy censoring such as $(0.20,0,20)$ the AV increases. As the sample size increase, the AV tends to decrease.

\noindent \textbf{Simulation results II:} Analogous configurations are considered for $k=3$ groups. It is observed from Tables \ref{table5}, \ref{table6}, \ref{table7} and \ref{table8} that the coverage probabilities of both the PB and fiducial methods remains close to the nominal level. The AV is relatively large for smaller sizes and decreases as the sample size increases. A noticeable difference in AV between two methods are observed for smaller sample sizes and it is similar when sample size increases. The choice of estimator has only a minor effect on both CP and AV for both methods, although some differences are observed. However, the choice between the sum and max of the esitimated variance has a noticeable effect on the results for the both methods.
The effect of censoring and sample size are consistent with those observed for $k=2$. 

\noindent \textbf{Simulation results III:} For $k=4$ groups, one parameter setting, five censoring proportions and six different sample size configurations are considered. It is observed from Tables \ref{table9} and \ref{table10} that the CP of the PB methods remain close to the nominal level, whereas those of the fiducial methods are generally higher than the nominal level. Similar patterns are observed with respect to the choice of estimator, sample size, the Sum and Max of the estimated variances and the censoring proportions as in the $k=3 $ case. The main difference is that the AV of the fiducial methods are at least twice as large as those of the corresponding PB methods.

\noindent \textbf{Simulation results IV:} Overall considering all methods, sample size configurations, parameter settings and censoring proportions, the PB method based on the unbiased estimator with the sum of estimated variances provides the most consistent overall performance. Furthermore, large parameter values, particularly large scale parameters gives wider confidence intervals and hence larger AV.

\begin{table}[h!]
	\centering
	\footnotesize
	\caption{Average coverage probability (CP) and average volume (AV) of the
		confidence intervals for ${(\mu_1, \mu_2)}=(0.35, 1.65)$ and
		${(\sigma_1, \sigma_2)}=(1.25, 0.75)$ with $N=10{,}000$ and $B=10{,}000$.}
	\label{table1}
	\setlength{\tabcolsep}{5pt}
	\renewcommand{\arraystretch}{1.15}
	
\begin{tabular}{c| c c c c c c c c c c}
		\toprule
		\multirow{5}{*}{
			$\begin{pmatrix}
				n_1\\
				n_2
			\end{pmatrix}$}
		& \multirow{5}{*}{$(p_1,p_2)$}
		& \multirow{5}{*}{$(q_1,q_2)$}
		& \multicolumn{8}{c}{Parametric Bootstrap}
		\\
		\cmidrule(lr){4-11}
		& & 
		& \multicolumn{4}{c}{Unbiased}
		& \multicolumn{4}{c}{MLE}
		\\
		\cmidrule(lr){4-7}
		\cmidrule(lr){8-11}
		& &
		& \multicolumn{2}{c}{Sum}
		& \multicolumn{2}{c}{Max}
		& \multicolumn{2}{c}{Sum}
		& \multicolumn{2}{c}{Max}
		\\
		\cmidrule(lr){4-5}
		\cmidrule(lr){6-7}
		\cmidrule(lr){8-9}
		\cmidrule(lr){10-11}
		& &
		& CP & AV
		& CP & AV
		& CP & AV
		& CP & AV
		\\
		\midrule
		\multirow{5}{*}{
			$\begin{pmatrix}
				30\\
				30
			\end{pmatrix}$}
		& (0.05, 0.05) & (0.05, 0.05) & 0.9505 & 1.0626 & 0.934 & 1.1077 & 0.9496 & 1.0557 & 0.9312 & 1.1079 \\
		& (0.10, 0.10,) & (0.10, 0.10) & 0.9445 & 1.1064 & 0.9248 & 1.151 & 0.9444 & 1.1055 & 0.9248 & 1.151 \\
		& (0.20, 0.20) & (0.20, 0.20) & 0.9499 & 1.1976 & 0.9346 & 1.2428 & 0.9498 & 1.1954 & 0.9342 & 1.2426 \\
		& (0.05, 0.10) & (0.10, 0.05) & 0.9512 & 1.0986 & 0.9331 & 1.1521 & 0.9453 & 1.1145 & 0.9279 & 1.1848 \\
\bottomrule
	\end{tabular}
	\vspace{2mm}
\end{table}

\begin{table}[h!]
	\centering
	\footnotesize
	\caption{Average coverage probability (CP) and average volume (AV) of the
		confidence intervals for ${(\mu_1, \mu_2)}=(0.35, 1.65)$ and
		${(\sigma_1, \sigma_2)}=(1.25, 0.75)$ with $N=10{,}000$ and $B=10{,}000$.}
	\label{table2}
	\setlength{\tabcolsep}{5pt}
	\renewcommand{\arraystretch}{1.15}
	
	\begin{tabular}{c| c c c c c c c c c c}
		\toprule
		\multirow{5}{*}{
			$\begin{pmatrix}
				n_1\\
				n_2
			\end{pmatrix}$}
		& \multirow{5}{*}{$(p_1,p_2)$}
		& \multirow{5}{*}{$(q_1,q_2)$}
		& \multicolumn{8}{c}{Fiducial}
		\\
		\cmidrule(lr){4-11}
		& & 
		& \multicolumn{4}{c}{Unbiased}
		& \multicolumn{4}{c}{MLE}
		\\
		\cmidrule(lr){4-7}
		\cmidrule(lr){8-11}
		& &
		& \multicolumn{2}{c}{Sum}
		& \multicolumn{2}{c}{Max}
		& \multicolumn{2}{c}{Sum}
		& \multicolumn{2}{c}{Max}
		\\
		\cmidrule(lr){4-5}
		\cmidrule(lr){6-7}
		\cmidrule(lr){8-9}
		\cmidrule(lr){10-11}
		& &
		& CP & AV
		& CP & AV
		& CP & AV
		& CP & AV
		\\
		\midrule
		
		\multirow{5}{*}{
			$\begin{pmatrix}
				30\\
				30
			\end{pmatrix}$}
		& (0.05,0.05) & (0.05,0.05)
		& 0.9690 & 1.1678 & 0.9690 & 1.1678
		& 0.9702 & 1.1738 & 0.9702 & 1.1738 \\
		
		& (0.10,0.10) & (0.10,0.10)
		& 0.9678 & 1.2258 & 0.9678 & 1.2258
		& 0.9677 & 1.2264 & 0.9677 & 1.2264 \\
		
		& (0.20,0.20) & (0.20,0.20)
		& 0.9737 & 1.3565 & 0.9737 & 1.3565
		& 0.9739 & 1.3588 & 0.9739 & 1.3588 \\
		
		& (0.05,0.10) & (0.10,0.05)
		& 0.9666 & 1.2074 & 0.9666 & 1.2074
		& 0.9649 & 1.2268 & 0.9649 & 1.2268 \\
\bottomrule
	\end{tabular}
\end{table}

\begin{table}[h!]
	\centering
	\footnotesize
	\caption{Average coverage probability (CP) and average volume (AV) of the
		confidence intervals for $(\mu_1, \mu_2)=(-0.75, -1.65),  (\sigma_1, \sigma_2)=(5.00, 3.75),$ with $N=10{,}000$ and $B=10{,}000$.}
	\label{table3}
	\setlength{\tabcolsep}{5pt}
	\renewcommand{\arraystretch}{1.15}
	
	\begin{tabular}{c| c c c c c c c c c c}
		\toprule
		\multirow{5}{*}{
			$\begin{pmatrix}
				n_1\\
				n_2
			\end{pmatrix}$}
		& \multirow{5}{*}{$(p_1,p_2)$}
		& \multirow{5}{*}{$(q_1,q_2)$}
		& \multicolumn{8}{c}{Parametric Bootstrap}
		\\
		\cmidrule(lr){4-11}
		& & 
		& \multicolumn{4}{c}{Unbiased}
		& \multicolumn{4}{c}{MLE}
		\\
		\cmidrule(lr){4-7}
		\cmidrule(lr){8-11}
		& &
		& \multicolumn{2}{c}{Sum}
		& \multicolumn{2}{c}{Max}
		& \multicolumn{2}{c}{Sum}
		& \multicolumn{2}{c}{Max}
		\\
		\cmidrule(lr){4-5}
		\cmidrule(lr){6-7}
		\cmidrule(lr){8-9}
		\cmidrule(lr){10-11}
		& &
		& CP & AV
		& CP & AV
		& CP & AV
		& CP & AV
		\\
		\midrule
		\multirow{5}{*}{
			$\begin{pmatrix}
				30\\
				30
			\end{pmatrix}$}
		& (0.05, 0.05) & (0.05, 0.05) & 0.9552 & 4.4792 & 0.9434 & 4.4276 & 0.9565 & 4.4266 & 0.9440 & 4.3991 \\
		& (0.10, 0.10) & (0.10, 0.10) & 0.9548 & 4.6612 & 0.9437 & 4.6041 & 0.9549 & 4.6548 & 0.9440 & 4.6005 \\
		& (0.20, 0.20) & (0.20, 0.20) & 0.9573 & 5.0258 & 0.9522 & 4.9803 & 0.9574 & 5.0090 & 0.9522 & 4.9712 \\
		& (0.05, 0.10) & (0.10, 0.05) & 0.9561 & 4.6155 & 0.9393 & 4.6106 & 0.9523 & 4.6439 & 0.9319 & 4.6845 \\
\bottomrule
	\end{tabular}
	\vspace{2mm}
\end{table}

\begin{table}[h!]
	\centering
	\footnotesize
	\caption{Average coverage probability (CP) and average volume (AV) of the
		confidence intervals for ${(\mu_1,\mu_2)}=(-0.75,-1.65)$ and
		${(\sigma_1,\sigma_2)}=(5.00,3.75)$ with $M=10{,}000$ and $B=10{,}000$.}
	\label{table4}
	\setlength{\tabcolsep}{5pt}
	\renewcommand{\arraystretch}{1.15}
	
	\begin{tabular}{c| c c c c c c c c c c}
		\toprule
		\multirow{5}{*}{
			$\begin{pmatrix}
				n_1\\
				n_2
			\end{pmatrix}$}
		& \multirow{5}{*}{$(p_1,p_2)$}
		& \multirow{5}{*}{$(q_1,q_2)$}
		& \multicolumn{8}{c}{Fiducial}
		\\
		
		\cmidrule(lr){4-11}
		
		& &
		& \multicolumn{4}{c}{Unbiased}
		& \multicolumn{4}{c}{MLE}
		\\
		
		\cmidrule(lr){4-7}
		\cmidrule(lr){8-11}
		
		& &
		& \multicolumn{2}{c}{Sum}
		& \multicolumn{2}{c}{Max}
		& \multicolumn{2}{c}{Sum}
		& \multicolumn{2}{c}{Max}
		\\
		
		\cmidrule(lr){4-5}
		\cmidrule(lr){6-7}
		\cmidrule(lr){8-9}
		\cmidrule(lr){10-11}
		
		& &
		& CP & AV
		& CP & AV
		& CP & AV
		& CP & AV
		\\
		
		\midrule
		
		\multirow{5}{*}{
			$\begin{pmatrix}
				30\\
				30
			\end{pmatrix}$}
		& (0.05,0.05) & (0.05,0.05)& 0.9747 & 5.0221 & 0.9747 & 5.0221
		& 0.9770 & 5.0347 & 0.9770 & 5.0347 \\
		& (0.10,0.10) & (0.10,0.10)& 0.9726 & 5.2700 & 0.9726 & 5.2700
		& 0.9729 & 5.2715 & 0.9729 & 5.2715 \\
		& (0.20,0.20) & (0.20,0.20)& 0.9792 & 5.8317 & 0.9792 & 5.8317
		& 0.9798 & 5.8371 & 0.9798 & 5.8371 \\
		& (0.05,0.10) & (0.10,0.05)& 0.9730 & 5.1723 & 0.9730 & 5.1723
		& 0.9711 & 5.2273 & 0.9711 & 5.2273 \\
\bottomrule
	\end{tabular}
\end{table}

\begin{table}[h!]
	\centering
	\footnotesize
	\caption{Average coverage probability (CP) and average volume (AV) of the
		confidence intervals for $(\mu_1, \mu_2, \mu_3)=(0,0,0)$ and
		$(\sigma_1, \sigma_2, \sigma_3)=(1,2,3)$ with $N=10{,}000$ and $B=10{,}000$.}
	\label{table5}
	\setlength{\tabcolsep}{4.5pt}
	\renewcommand{\arraystretch}{1.1}
	
	\begin{tabular}{c| c c c c c c c c c c}
		\toprule
		\multirow{5}{*}{
			$\begin{pmatrix}
				n_1\\
				n_2\\
				n_3
			\end{pmatrix}$}
		& \multirow{5}{*}{$(p_1,p_2,p_3)$}
		& \multirow{5}{*}{$(q_1,q_2,q_3)$}
		& \multicolumn{8}{c}{Parametric Bootstrap}
		\\
		\cmidrule(lr){4-11}
		& & 
		& \multicolumn{4}{c}{Unbiased}
		& \multicolumn{4}{c}{MLE}
		\\
		\cmidrule(lr){4-7}
		\cmidrule(lr){8-11}
		& &
		& \multicolumn{2}{c}{Sum}
		& \multicolumn{2}{c}{Max}
		& \multicolumn{2}{c}{Sum}
		& \multicolumn{2}{c}{Max}
		\\
		\cmidrule(lr){4-5}
		\cmidrule(lr){6-7}
		\cmidrule(lr){8-9}
		\cmidrule(lr){10-11}
		& &
		& CP & AV
		& CP & AV
		& CP & AV
		& CP & AV
		\\
		\midrule
		\multirow{5}{*}{
			$\begin{pmatrix}
				30\\
				30\\
				30
			\end{pmatrix}$}
		& $(0.05,0.05,0.05)$ & $(0.05,0.05,0.05)$ & 0.9391 & 19.7874 & 0.9341 & 24.8745 & 0.9372 & 20.1442 & 0.9334 & 25.2632 \\
		& $(0.10,0.10,0.10)$ & $(0.10,0.10,0.10)$ & 0.9451 & 22.8498 & 0.9394 & 28.5063 & 0.9452 & 22.8915 & 0.9393 & 28.5544 \\
		& $(0.20,0.20,0.20)$ & $(0.20,0.20,0.20)$ & 0.9406 & 30.3231 & 0.9362 & 37.5530 & 0.9400 & 30.4748& 0.9358 & 37.7156 \\
		& $(0.05,0.10,0.20)$ & $(0.20,0.10,0.05)$ & 0.9459 & 21.7055 & 0.9400 & 26.4851 & 0.9475 & 20.7000 & 0.9428 & 25.0489 \\
\bottomrule
	\end{tabular}
	\vspace{2mm}
\end{table}

\begin{table}[h!]
	\centering
	\footnotesize
	\caption{Average coverage probability (CP) and average volume (AV) of the
		confidence intervals for ${(\mu_1,\mu_2,\mu_3)}=(0,0,0)$ and
		${(\sigma_1,\sigma_2,\sigma_3)}=(1,2,3)$ with $M=10{,}000$ and $B=10{,}000$.}
	\label{table6}
	\setlength{\tabcolsep}{4pt}
	\renewcommand{\arraystretch}{1.15}
	
	\begin{tabular}{c|cccccccccc}
		\toprule
		
		\multirow{5}{*}{$\begin{pmatrix}
				n_1\\
				n_2\\
				n_3
			\end{pmatrix}$}
		& \multirow{5}{*}{$(p_1,p_2,p_3)$}
		& \multirow{5}{*}{$(q_1,q_2,q_3)$}
		& \multicolumn{8}{c}{Fiducial}
		\\
		
		\cmidrule(lr){4-11}
		
		& &
		& \multicolumn{4}{c}{Unbiased}
		& \multicolumn{4}{c}{MLE}
		\\
		
		\cmidrule(lr){4-7}
		\cmidrule(lr){8-11}
		
		& &
		& \multicolumn{2}{c}{Sum}
		& \multicolumn{2}{c}{Max}
		& \multicolumn{2}{c}{Sum}
		& \multicolumn{2}{c}{Max}
		\\
		
		\cmidrule(lr){4-5}
		\cmidrule(lr){6-7}
		\cmidrule(lr){8-9}
		\cmidrule(lr){10-11}
		
		& &
		& CP & AV
		& CP & AV
		& CP & AV
		& CP & AV
		\\
		
		\midrule
		\multirow{5}{*}{$\begin{pmatrix}
				30\\
				30\\
				30
			\end{pmatrix}$}
		& (0.05,0.05,0.05) & (0.05,0.05,0.05) & 0.9682 & 24.6821 & 0.9607 & 25.5699
		& 0.9687 & 25.6069 & 0.9610 & 26.2821 \\
		& (0.10,0.10,0.10) & (0.10,0.10,0.10) & 0.9665 & 28.8868 & 0.9622 & 29.9386
		& 0.9666 & 29.0037 & 0.9623 & 30.0286 \\
		& (0.20,0.20,0.20) & (0.20,0.20,0.20) & 0.9673 & 40.6087 & 0.9604 & 42.0354
		& 0.9676 & 41.0492 & 0.9601 & 42.3695 \\
		& (0.05,0.10,0.20) & (0.20,0.10,0.05) & 0.9705 & 28.0107 & 0.9650 & 29.1184
		& 0.9725 & 27.4222 & 0.9683 & 28.6658 \\
\bottomrule
	\end{tabular}
\end{table}

\begin{table}[h!]
	\centering
	\footnotesize
	\caption{Average coverage probabilities (CP) and average volume (AV) of the
		confidence intervals for $(\mu_1, \mu_2, \mu_3)=(0.5,1,2)$ and
		$(\sigma_1, \sigma_2, \sigma_3)=(1,1.5,2)$ with $N=10{,}000$ and $B=10{,}000$.}
	\label{table7}
	\setlength{\tabcolsep}{4pt}
	\renewcommand{\arraystretch}{1.1}
	\begin{tabular}{c| c c c c c c c c c c}
		\toprule
		\multirow{5}{*}{
			$\begin{pmatrix}
				n_1\\
				n_2\\
				n_3
			\end{pmatrix}$}
		& \multirow{5}{*}{$(p_1,p_2,p_3)$}
		& \multirow{5}{*}{$(q_1,q_2,q_3)$}
		& \multicolumn{8}{c}{Parametric Bootstrap}
		\\
		\cmidrule(lr){4-11}
		& & 
		& \multicolumn{4}{c}{Unbiased}
		& \multicolumn{4}{c}{MLE}
		\\
		\cmidrule(lr){4-7}
		\cmidrule(lr){8-11}
		& &
		& \multicolumn{2}{c}{Sum}
		& \multicolumn{2}{c}{Max}
		& \multicolumn{2}{c}{Sum}
		& \multicolumn{2}{c}{Max}
		\\
		\cmidrule(lr){4-5}
		\cmidrule(lr){6-7}
		\cmidrule(lr){8-9}
		\cmidrule(lr){10-11}
		& &
		& CP & AV
		& CP & AV
		& CP & AV
		& CP & AV
		\\
		\midrule
		\multirow{5}{*}{
			$\begin{pmatrix}
				30\\
				30\\
				30
			\end{pmatrix}$}
		& $(0.05,0.05,0.05)$ & $(0.05,0.05,0.05)$ & 0.9512 & 6.9417 & 0.9388 & 8.1204 & 0.9494 & 6.8656 & 0.9375 & 8.1199 \\
		& $(0.10,0.10,0.10)$ & $(0.10,0.10,0.10)$ & 0.9477 & 7.9107 & 0.9348 & 9.191 & 0.9480 & 7.8998 & 0.9346 & 9.1901\\
		& $(0.20,0.20,0.20)$ & $(0.20,0.20,0.20)$ & 0.9476 & 10.2481 & 0.9351 & 11.8507 & 0.9475 & 10.2198 & 0.9349 & 11.849\\
		& $(0.05,0.10,0.20)$ & $(0.20,0.10,0.05)$ & 0.9503 & 7.6329 & 0.9387 & 8.4287 & 0.9538 & 7.3286 & 0.9434 & 7.9900\\
\bottomrule
	\end{tabular}
\end{table}

\begin{table}[h!]
	\centering
	\footnotesize
	\caption{Average coverage probability (CP) and average volume (AV) of the
		confidence intervals for ${(\mu_1,\mu_2,\mu_3)}=(0.5,1,2)$ and
		${(\sigma_1,\sigma_2,\sigma_3)}=(1,1.5,2)$ with $M=10{,}000$ and $B=10{,}000$.}
	\label{table8}
	\setlength{\tabcolsep}{4pt}
	\renewcommand{\arraystretch}{1.15}
	
	\begin{tabular}{c|cccccccccc}
		\toprule
		
		\multirow{5}{*}{$\begin{pmatrix}
				n_1\\
				n_2\\
				n_3
			\end{pmatrix}$}
		& \multirow{5}{*}{$(p_1,p_2,p_3)$}
		& \multirow{5}{*}{$(q_1,q_2,q_3)$}
		& \multicolumn{8}{c}{Fiducial}
		\\
		\cmidrule(lr){4-11}
		& &
		& \multicolumn{4}{c}{Unbiased}
		& \multicolumn{4}{c}{MLE}
		\\
		\cmidrule(lr){4-7}
		\cmidrule(lr){8-11}
		& &
		& \multicolumn{2}{c}{Sum}
		& \multicolumn{2}{c}{Max}
		& \multicolumn{2}{c}{Sum}
		& \multicolumn{2}{c}{Max}
		\\
		\cmidrule(lr){4-5}
		\cmidrule(lr){6-7}
		\cmidrule(lr){8-9}
		\cmidrule(lr){10-11}
		& &
		& CP & AV
		& CP & AV
		& CP & AV
		& CP & AV
		\\
		\midrule
		\multirow{5}{*}{$\begin{pmatrix}
				30\\
				30\\
				30
			\end{pmatrix}$}
		& (0.05,0.05,0.05) & (0.05,0.05,0.05) & 0.9769 & 9.6390 & 0.9733 & 10.0231
		& 0.9781 & 9.8307 & 0.9744 & 10.1735 \\
		& (0.10,0.10,0.10) & (0.10,0.10,0.10) & 0.9761 & 11.2780 & 0.9741 & 11.7269
		& 0.9763 & 11.3028 & 0.9743 & 11.7462 \\
		& (0.20,0.20,0.20) & (0.20,0.20,0.20) & 0.9800 & 15.7786 & 0.9769 & 16.4284
		& 0.9804 & 15.8769 & 0.9774 & 16.5041 \\
		& (0.05,0.10,0.20) & (0.20,0.10,0.05) & 0.9798 & 11.2241 & 0.9795 & 11.6562
		& 0.9818 & 11.0915 & 0.9819 & 11.5814 \\
\bottomrule
	\end{tabular}
\end{table}

\begin{table}[h!]
	\centering 
	\caption{Average coverage probability (CP) and average volume (AV) of the
		confidence intervals for $(\mu_1,\mu_2, \mu_3, \mu_4)=(0.00,0.50,1.00,1.50)$ and
		$(\sigma_1, \sigma_2, \sigma_3, \sigma_4)=(1.00,1.10,1.20,1.30)$ with $N=10{,}000$ and $B=10{,}000$.} 
	\label{table9} 
	\setlength{\tabcolsep}{4.5pt} 
	\renewcommand{\arraystretch}{1.2} 
	\small
	\resizebox{\textwidth}{!}{
		\begin{tabular}{c| c c c c c c c c c c} 
			\toprule 
			\multirow{5}{*}{$\begin{pmatrix}
					n_1\\
					n_2\\
					n_3\\
					n_4
				\end{pmatrix}$} & \multirow{5}{*}{$(p_1, p_2, p_3, p_4)$} & \multirow{5}{*}{$(q_1, q_2, q_3, q_4)$} & \multicolumn{8}{c}{Parametric Bootstrap} \\ 
			\cmidrule(lr){4-11} 
			& & & \multicolumn{4}{c}{Unbiased} & \multicolumn{4}{c}{MLE} \\ 
			\cmidrule(lr){4-7} \cmidrule(lr){8-11} 
			& & & \multicolumn{2}{c}{Sum} & \multicolumn{2}{c}{Max} & \multicolumn{2}{c}{Sum} & \multicolumn{2}{c}{Max} \\ 
			\cmidrule(lr){4-5} \cmidrule(lr){6-7} \cmidrule(lr){8-9} \cmidrule(lr){10-11} 
			& & & CP & AV & CP & AV & CP & AV & CP & AV \\ 
			\midrule 
			\multirow{5}{*}{$\begin{pmatrix}
					30\\
					30\\
					30\\
					30
				\end{pmatrix}$} 
			& (0.05, 0.05, 0.05, 0.05) & (0.05, 0.05, 0.05, 0.05) & 0.9663 & 12.5786 & 0.9821 & 10.6348 & 0.9675 & 11.7572 & 0.9835 & 10.2883 \\ 
			& (0.10, 0.10, 0.10, 0.10) & (0.10, 0.10, 0.10, 0.10) & 0.9682 & 16.6223 & 0.9815 & 14.7262 & 0.9682 & 16.4972 & 0.9817 & 14.6744 \\ 
			& (0.05, 0.10, 0.15, 0.20) & (0.05, 0.10, 0.15, 0.20) & 0.9666 & 18.2003 & 0.9817 & 16.0152 & 0.9689 & 17.6251 & 0.9833 & 15.9006 \\ 
			& (0.10, 0.15, 0.20, 0.25) & (0.10, 0.15, 0.20, 0.25) & 0.9655 & 23.6720 & 0.9806 & 22.2397 & 0.9632 & 22.6099 & 0.9789 & 21.6208 \\
\bottomrule 
		\end{tabular}%
	}
\end{table}

\begin{table}[h!] 
	\centering 
	\caption{Average coverage probability (CP) and average volume (AV) of the
		confidence intervals for $(\mu_1,\mu_2, \mu_3, \mu_4)=(0.00,0.50,1.00,1.50)$ and
		$(\sigma_1, \sigma_2, \sigma_3, \sigma_4)=(1.00,1.10,1.20,1.30)$ with $N=10{,}000$ and $B=10{,}000$.} 
	\label{table10} 
	\setlength{\tabcolsep}{4.5pt} 
	\renewcommand{\arraystretch}{1.2} 
	\small
	\resizebox{\textwidth}{!}{
		\begin{tabular}{c| c c c c c c c c c c} 
			\toprule 
			\multirow{5}{*}{$\begin{pmatrix}
					n_1\\
					n_2\\
					n_3\\
					n_4
				\end{pmatrix}$} & \multirow{5}{*}{$(p_1, p_2, p_3, p_4)$} & \multirow{5}{*}{$(q_1, q_2, q_3, q_4)$} & \multicolumn{8}{c}{Fiducial} \\ 
			\cmidrule(lr){4-11} 
			& & & \multicolumn{4}{c}{Unbiased} & \multicolumn{4}{c}{MLE} \\ 
			\cmidrule(lr){4-7} \cmidrule(lr){8-11} 
			& & & \multicolumn{2}{c}{Sum} & \multicolumn{2}{c}{Max} & \multicolumn{2}{c}{Sum} & \multicolumn{2}{c}{Max} \\ 
			\cmidrule(lr){4-5} \cmidrule(lr){6-7} \cmidrule(lr){8-9} \cmidrule(lr){10-11} 
			& & & CP & AV & CP & AV & CP & AV & CP & AV \\ 
			\midrule 
			\multirow{5}{*}{$\begin{pmatrix}
					30\\
					30\\
					30\\
					30
				\end{pmatrix}$} 
			& (0.05, 0.05, 0.05, 0.05) & (0.05, 0.05, 0.05, 0.05) & 0.9900 & 32.8600 & 0.9964 & 35.1612 & 0.9916 &  33.3719 & 0.9977 & 35.5461\\ 
			& (0.10, 0.10, 0.10, 0.10) & (0.10, 0.10, 0.10, 0.10) & 0.9932 & 46.7033 & 0.9972 &	50.0957 & 0.9935 & 46.7921 & 0.9975	& 50.1629\\ 
			& (0.05, 0.10, 0.15, 0.20) & (0.05, 0.10, 0.15, 0.20) & 0.9938 & 57.4338 & 0.9977 &	62.1655 & 0.9949 & 58.0326 & 0.9979 & 62.7772\\
			& (0.10, 0.15, 0.20, 0.25) & (0.10, 0.15, 0.20, 0.25) & 0.9938 & 80.8981 & 0.9981 &	87.7853	& 0.9945 & 84.8668 & 0.9982 & 90.8107\\
\bottomrule 
		\end{tabular}%
	}
\end{table}

\clearpage

\section{Real Data Analysis}\label{VI}
In this section, we illustrate the proposed simultaneous confidence interval methods using two real data sets. First data consists of breaking strength of jute fibers and was taken from \cite{xia2009study}. The data consists of $30$ observations for each of the three gauge lengths: $5$ mm, $10$ mm and $20$ mm. The data are given in table \ref{T:6}. The second data set consists of remission duration for patients receiving different drugs and was taken from \cite{malekzadeh2014comparing}. The data are presented in table \ref{T:10}. For both data set, the suitability of the two-parameter exponential distribution was verified for each group using the one -sample Kolmogorov-Smirnov test. The resulting K-S P-value are presented in tables \ref{T:7} and \ref{T:11}.

To investigate the applicability of the proposed methods in presence of censoring, doubly Type-II censored samples were generated from the complete data by considering the specified censoring scheme. The proposed parametric bootstrap SCI method and the FGPQ-based SCI method were then employed to construct simultaneous confidence intervals for all pairwise differences of means $\theta_{i}=\mu_i+\sigma_i~i=1,2,\dots,k$. The critical value $q^{n,PB}_{\alpha}$ and $F^{n}_{\alpha}$ were obtained using the procedures described in Sections \ref{II} and \ref{IV}. Using these critical values, two-sided SCIs were constructed for all pairwise differences of $\theta_{i}$. The resulting PB SCIs and FG SCIs are reported in the following tables. 

The length of the resulting SCIs were also compared to assess the relative efficiency of the two methods. The method giving shorter intervals is preferred. Overall, the result shows that the proposed two methods can be effectively applied to real data under censoring to obtain SCIs for the pairwise differences of means.

\begin{table}[H]
	\centering
	\caption{The breaking strengths of jute fiber at different lengths.}
	\label{T:6}
	\renewcommand{\arraystretch}{1.2}
	\setlength{\tabcolsep}{6pt}
	\begin{tabular}{c l}
		\toprule
		\textbf{Length} & \textbf{Breaking strengths} \\
		\midrule
		5 mm 
		& 566.31, 270.79, 516.48, 823.03, 226.53, 367.70, 185.42, 441.87, 618.57, 546.11 \\
		& 268.20, 315.33, 809.23, 218.86, 583.97, 304.84, 129.08, 537.45, 496.28, 167.87 \\
		& 306.99, 178.25, 370.02, 168.20, 554.61, 360.80, 260.97, 254.29, 495.51, 187.68 \\
		\midrule
		10 mm 
		& 693.73, 704.66, 323.83, 778.17, 123.06, 637.66, 383.43, 151.48, 108.94, 50.16 \\
		& 671.49, 183.16, 257.44, 727.23, 291.27, 101.15, 376.42, 163.40, 141.38, 700.74 \\
		& 262.90, 353.24, 422.11, 43.93, 590.48, 212.13, 303.90, 506.60, 530.55, 177.25 \\
		\midrule
		20 mm 
		& 71.46, 419.02, 284.64, 585.57, 456.60, 113.85, 187.85, 688.16, 662.66, 45.58 \\
		& 578.62, 756.70, 594.29, 166.49, 99.72, 707.36, 765.14, 187.13, 145.96, 350.70 \\
		& 547.44, 116.99, 375.81, 581.60, 119.86, 48.01, 200.16, 36.75, 244.53, 83.50 \\
		\bottomrule
	\end{tabular}
\end{table}

\begin{figure}[htbp]
	\centering
	\begin{minipage}{0.30\textwidth}
		\centering
		\includegraphics[height=4.75cm,width=4.75cm]{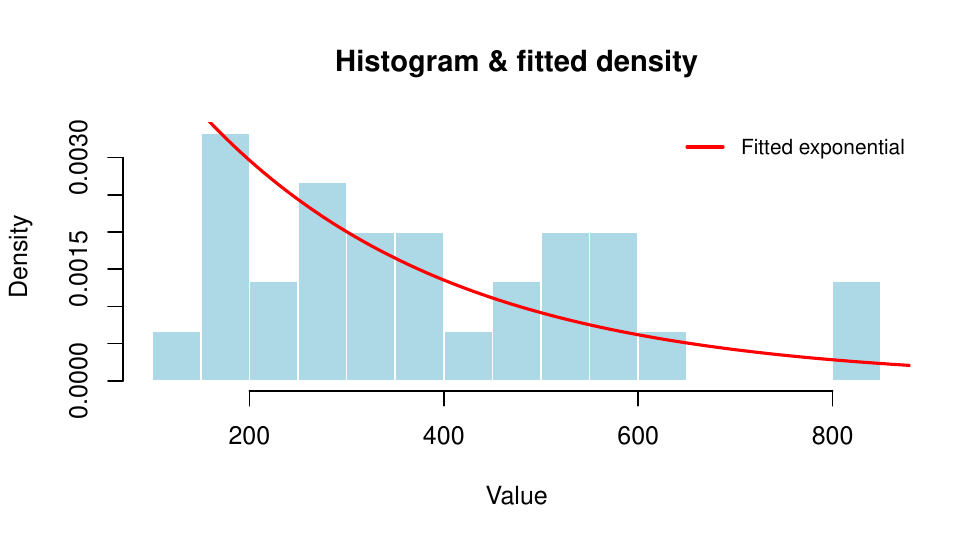}
	\end{minipage}
	\hspace{0.02\textwidth}
	\begin{minipage}{0.30\textwidth}
		\centering
		\includegraphics[height=4.75cm,width=4.75cm]{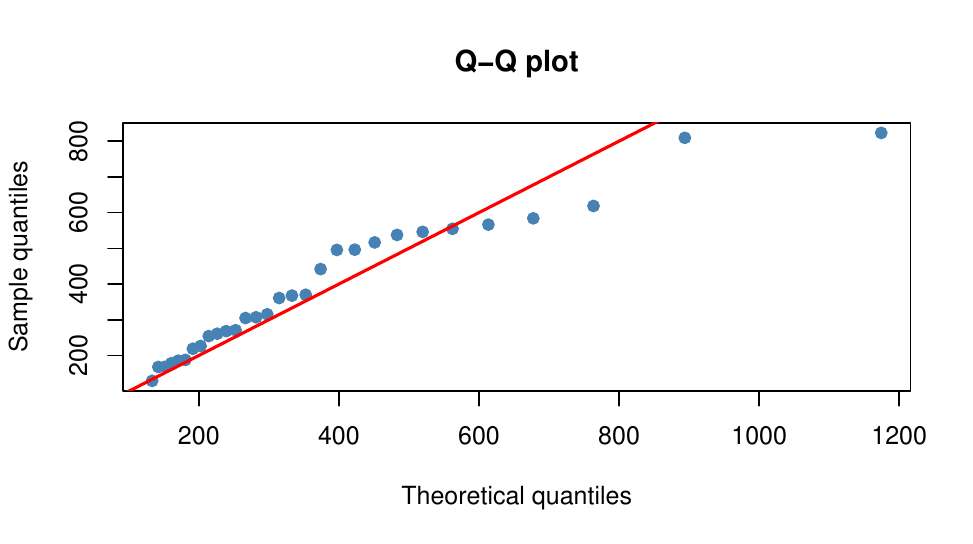}
	\end{minipage}
	\hspace{0.02\textwidth}
	\begin{minipage}{0.30\textwidth}
		\centering
		\includegraphics[height=4.75cm,width=4.75cm]{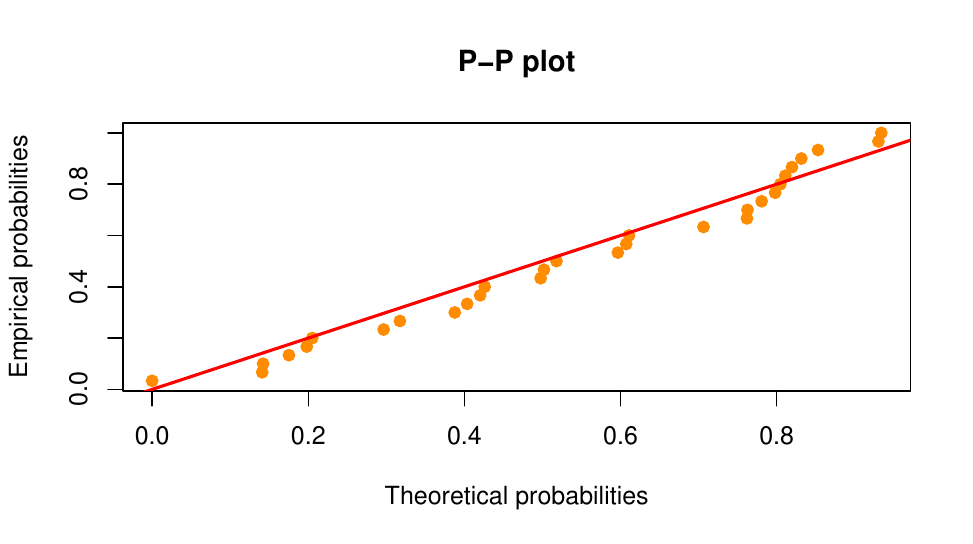}
	\end{minipage}
	\vspace{0.00cm}
	\vspace{0.25cm}
	\begin{minipage}{0.30\textwidth}
		\centering
		\includegraphics[height=4.75cm,width=4.75cm]{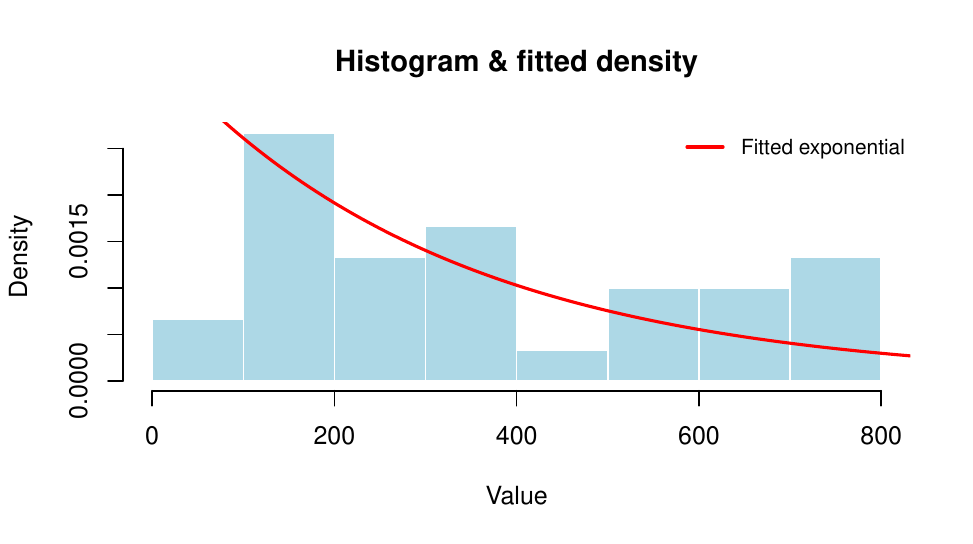}
	\end{minipage}
	\hspace{0.02\textwidth}
	\begin{minipage}{0.30\textwidth}
		\centering
		\includegraphics[height=4.75cm,width=4.75cm]{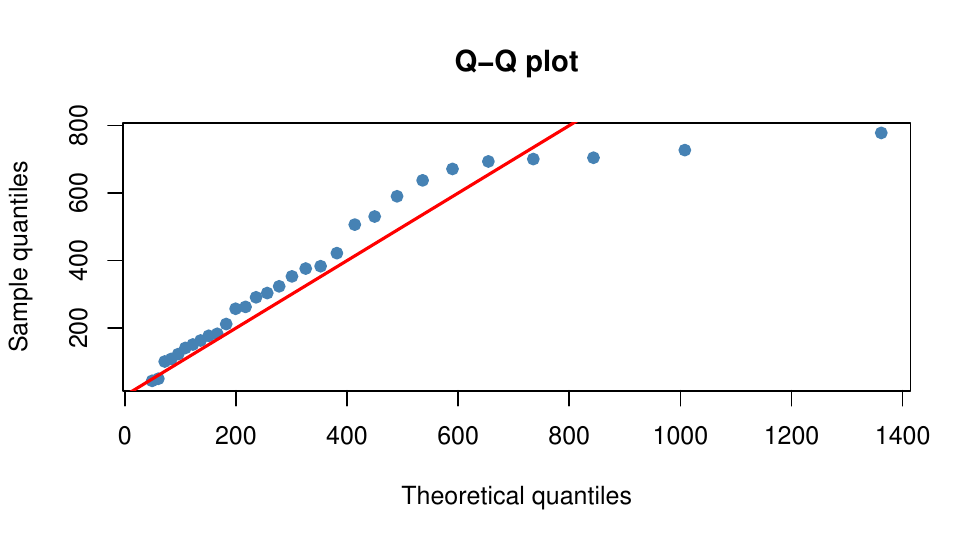}
	\end{minipage}
	\hspace{0.02\textwidth}
	\begin{minipage}{0.30\textwidth}
		\centering
		\includegraphics[height=4.75cm,width=4.75cm]{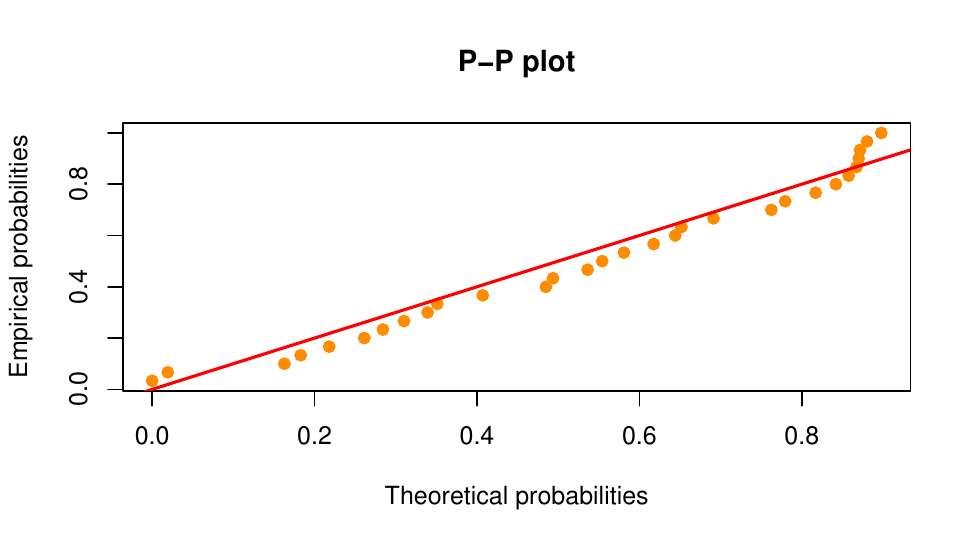}
	\end{minipage}
	\vspace{0.00cm}
	\vspace{0.00cm}
	\begin{minipage}{0.30\textwidth}
		\centering
		\includegraphics[height=4.75cm,width=4.75cm]{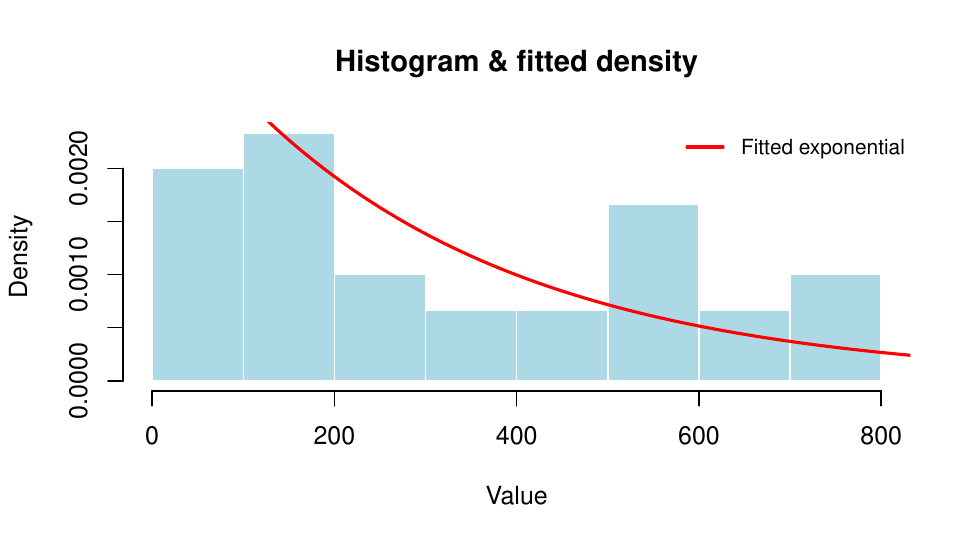}
	\end{minipage}
	\hspace{0.02\textwidth}
	\begin{minipage}{0.30\textwidth}
		\centering
		\includegraphics[height=4.75cm,width=4.75cm]{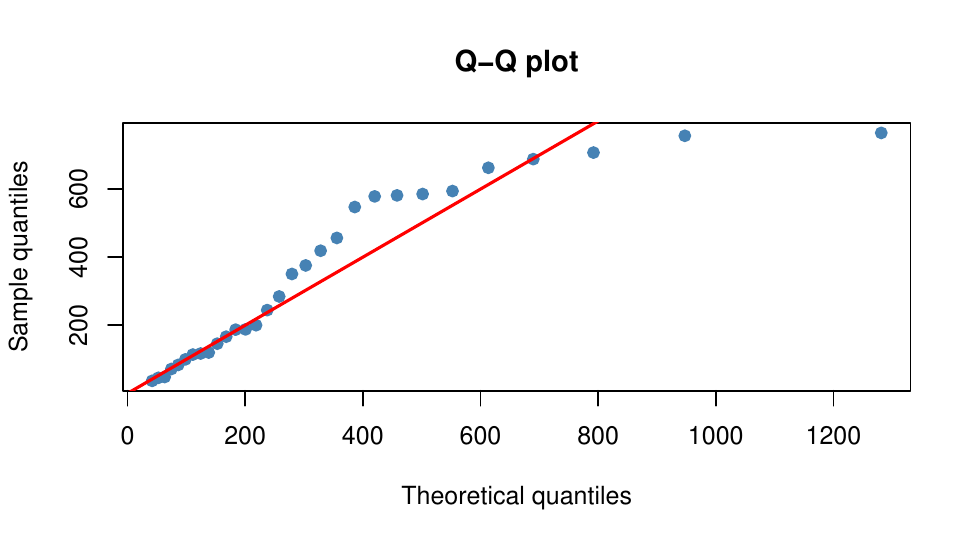}
	\end{minipage}
	\hspace{0.02\textwidth}
	\begin{minipage}{0.30\textwidth}
		\centering
		\includegraphics[height=4.75cm,width=4.75cm]{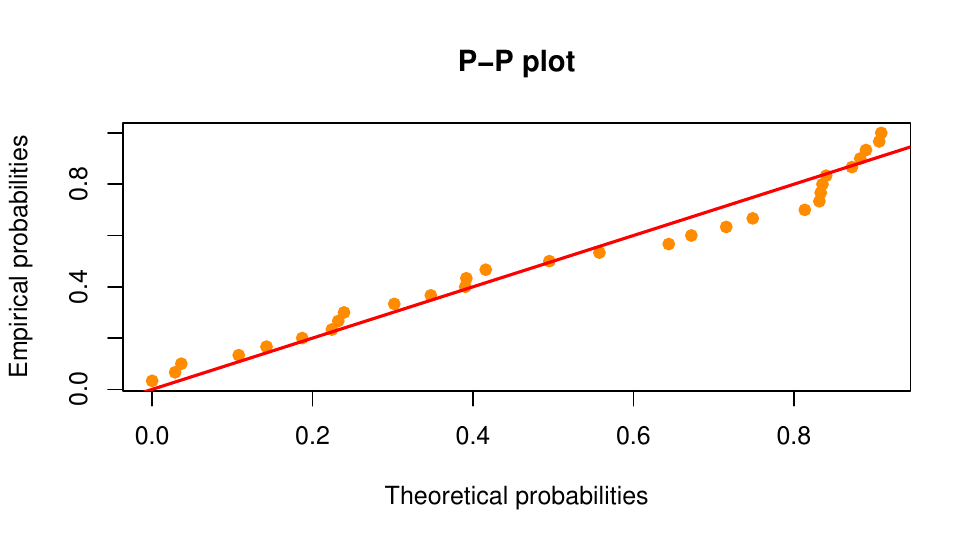}
	\end{minipage}
	\vspace{0.05cm}
	\begin{minipage}{0.70\textwidth}
		\centering
		\includegraphics[height=8cm,width=15cm]{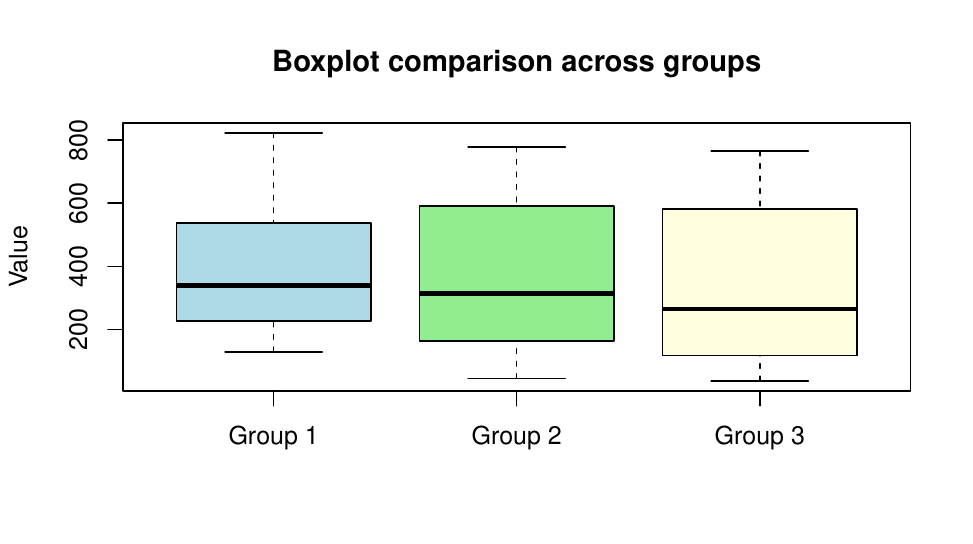}
	\end{minipage}
	
	\caption{Graphical representations of the three jute fiber datasets (Jute 1--3) and the corresponding boxplot.}
	
	\label{}
\end{figure}

\begin{table}[H]
	\centering
	\caption{Summary statistics and maximum likelihood estimates of $\theta_{i}$ for the complete jute fiber breaking strength data.}
	\label{T:7}
	\begin{tabular}{c c c c c c}
		\toprule
		\textbf{Gauge Length} & \textbf{Sample Size} & $\boldsymbol{\mu}$ & $\boldsymbol{\sigma}$ & $\boldsymbol{\theta}$ & \textbf{K-S statistic} \\
		\midrule
		5 mm  & 30 & 129.08 & 255.2947 & 384.375 & 0.6568 \\
		10 mm & 30 & 43.93  & 321.7997 & 365.73  & 0.7517 \\
		20 mm & 30 & 36.75  & 303.99   & 340.74  & 0.4908 \\
		\bottomrule
	\end{tabular}
\end{table}

\vspace{1em}

\begin{table}[h]
	\centering
	\caption{Parametric bootstrap simultaneous confidence intervals for differences of means under 10\% doubly censored data}
	\label{tab:sci}
	\small
	\begin{tabular}{l ccc ccc}
		\toprule
		& \multicolumn{3}{c}{\textbf{Unbiased (Sum)}} & \multicolumn{3}{c}{\textbf{MLE (Sum)}} \\
		\cmidrule(lr){2-4} \cmidrule(lr){5-7}
		\textbf{Parameter} & Lower & Upper & Length & Lower & Upper & Length \\
		\midrule
		$\theta_1-\theta_2$ & -186.3946 & 177.5651 & 363.9597 & -184.4705 & 175.9961 & 360.4666 \\
		$\theta_1-\theta_3$ &-163.1314 &209.0660 &372.1974 &-161.1434 & 207.4818 &368.6252  \\
		$\theta_2-\theta_3$ &-182.6714 & 237.4355 & 420.1068 & -180.6311 & 235.4438 & 416.0749  \\
		\bottomrule
	\end{tabular}
	
	\vspace{0.5cm}
	
	\begin{tabular}{l ccc ccc}
		\toprule
		& \multicolumn{3}{c}{\textbf{Unbiased (Max)}} & \multicolumn{3}{c}{\textbf{MLE (Max)}} \\
		\cmidrule(lr){2-4} \cmidrule(lr){5-7}
		\textbf{Parameter} & Lower & Upper & Length & Lower & Upper & Length \\
		\midrule
		$\theta_1-\theta_2$ & -179.0341 & 170.2046 & 349.2387 & -177.5929 & 169.1185 & 346.7114 \\
		$\theta_1-\theta_3$ & -157.7587 & 203.6933 & 361.4521 & -156.2491 & 202.5874 & 358.8365  \\
		$\theta_2-\theta_3$ & -180.6311 & 235.4438 &416.0749&  -152.0119& 206.8246& 358.8365\\
		\bottomrule
	\end{tabular}
\end{table}

\begin{table}[h]
	\centering
	\caption{Fiducial generalized  simultaneous confidence intervals for differences of means under 10\% Doubly censored data}
	\label{}
	\small
	\begin{tabular}{l ccc ccc}
		\toprule
		& \multicolumn{3}{c}{\textbf{Unbiased (Sum)}} & \multicolumn{3}{c}{\textbf{MLE (Sum)}} \\
		\cmidrule(lr){2-4} \cmidrule(lr){5-7}
		\textbf{Parameter} & Lower & Upper & Length & Lower & Upper & Length \\
		\midrule
		$\theta_1-\theta_2$ &-217.2833 &208.4538 &425.7371 &-216.6731 &208.1987 &424.8717\\
		$\theta_1-\theta_3$ &-194.7193 &240.6538 &435.3731 &-194.0749 &240.4132 &434.4881\\
		$\theta_2-\theta_3$ &-218.3252 &273.0893 &491.4146
		&-217.8015 &272.6142 &490.4157\\
		\bottomrule
	\end{tabular}
	
	\vspace{0.5cm}
	
	\begin{tabular}{l ccc ccc}
		\toprule
		& \multicolumn{3}{c}{\textbf{Unbiased (Max)}} & \multicolumn{3}{c}{\textbf{MLE (Max)}} \\
		\cmidrule(lr){2-4} \cmidrule(lr){5-7}
		\textbf{Parameter} & Lower & Upper & Length & Lower & Upper & Length \\
		\midrule
		$\theta_1-\theta_2$ &-222.6942 &213.8647 &436.5589 &-228.9228 &220.4484 &449.3713\\
		$\theta_1-\theta_3$ &-202.9457 &248.8803 &451.8260 &-209.3741 &255.7124 &465.0865\\
		$\theta_2-\theta_3$ &-198.5310 &253.2950 &451.8260 &-205.1369 &259.9496 &465.0865\\
		\bottomrule
	\end{tabular}
\end{table}

\clearpage

\begin{table}[H]
	\centering
	\caption{Remission duration (in months) for patients under four different drugs.}
	\label{T:10}
	\renewcommand{\arraystretch}{1.2}
	\setlength{\tabcolsep}{6pt}
	\begin{tabular}{c l}
		\toprule
		\textbf{Drug} & \textbf{Remission Duration} \\
		\midrule
		Test drug 1
		& 1.034, 2.344, 1.266, 1.563, 1.169, 4.118, 1.013, 1.509, 1.109, 1.965, 5.136, 1.533, \\
		& 1.716, 2.778, 2.546, 2.626, 3.413, 1.929, 2.061, 2.951\\
		\midrule
		Test drug 2
		& 5.115, 4.498, 4.617, 4.651, 4.533, 4.513, 7.641, 5.971,
		12.130, 4.699, 4.914, 17.169,  \\
		& 5.497, 11.332, 18.922, 13.712, 6.309, 10.086, 9.293, 11.787\\
		\midrule
		Control drug 1
		& 2.214, 4.976, 8.154, 2.686, 2.271, 3.139, 2.214, 4.480,
		8.847, 2.239, 3.473, 2.761, \\
		& 2.833, 2.381, 3.548, 2.414, 2.832, 5.551, 3.376, 2.968 \\
		\midrule
		Control drug 2
		& 4.158, 4.025, 5.170, 11.909, 4.912, 4.629, 3.955, 6.735,
		3.140, 12.446, 8.777, 6.321,  \\
		& 3.256, 8.250, 3.759, 5.205, 3.071, 3.147, 9.773, 10.218 \\
		\bottomrule
	\end{tabular}
\end{table}

\begin{table}[H]
	\centering
	\caption{Summary statistics and maximum likelihood estimates of $\theta_{i}$ for the complete drug data.}
	\label{T:11}
	\begin{tabular}{c c c c c c}
		\toprule
		\textbf{Drug} & \textbf{Sample Size} & $\boldsymbol{\mu}$ & $\boldsymbol{\sigma}$ & $\boldsymbol{\theta}$ & \textbf{K-S statistic} \\
		\midrule
		Test drug 1    & 20 & 1.0130 & 1.1756 & 2.1886 & 0.9869 \\
		Test drug 2    & 20 & 4.4980 & 3.8715 & 8.3695 & 0.1300 \\
		Control drug 1 & 20 & 2.2140 & 1.4539 & 3.6679 & 0.5989 \\
		Control drug 2 & 20 & 3.0710 & 3.0718 & 6.1428 & 0.7669 \\
		\bottomrule
	\end{tabular}
\end{table}

\begin{figure}[htbp]
	\centering
	\begin{minipage}{0.30\textwidth}
		\centering
		\includegraphics[height=3.75cm,width=4.75cm]{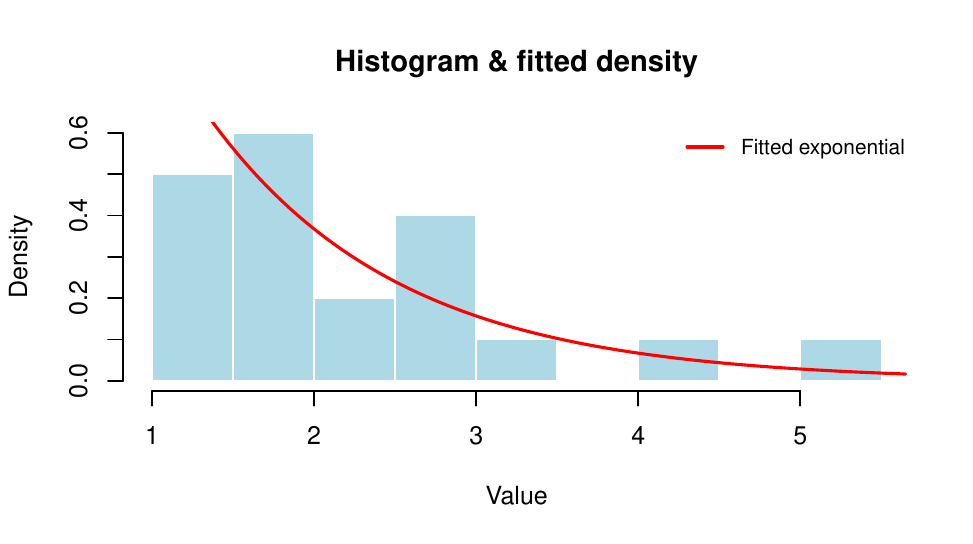}
	\end{minipage}
	\hspace{0.02\textwidth}
	\begin{minipage}{0.30\textwidth}
		\centering
		\includegraphics[height=3.75cm,width=4.75cm]{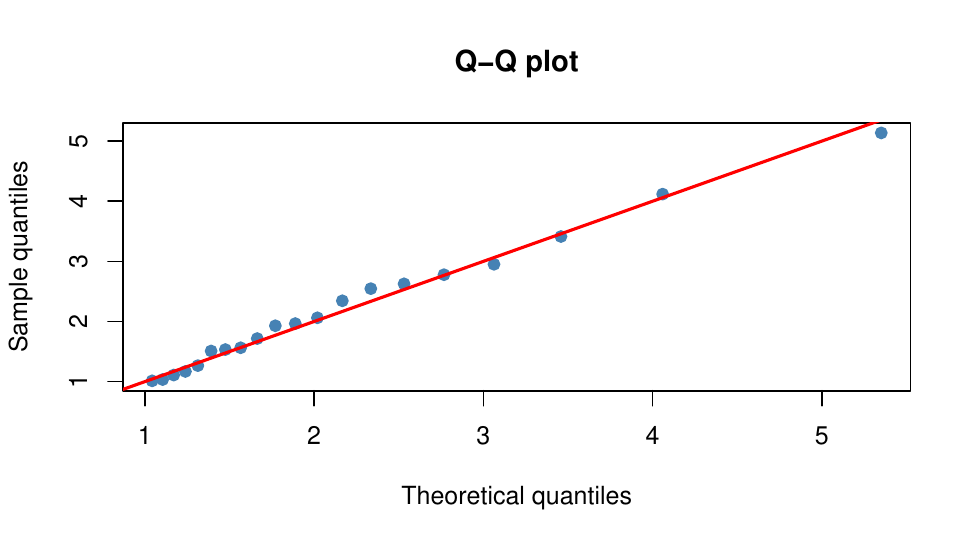}
	\end{minipage}
	\hspace{0.02\textwidth}
	\begin{minipage}{0.30\textwidth}
		\centering
		\includegraphics[height=3.75cm,width=4.75cm]{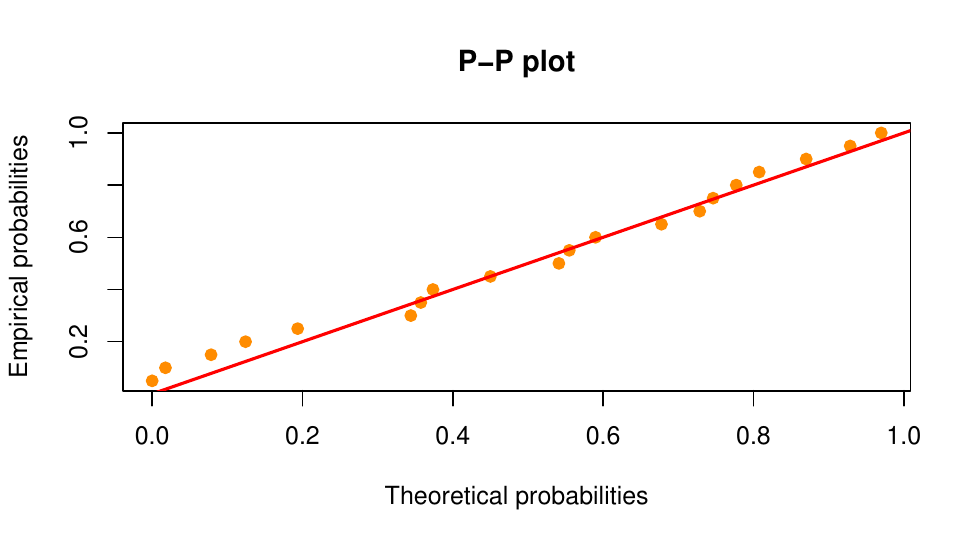}
	\end{minipage}
	\vspace{0.00cm}
	\vspace{0.25cm}
	\begin{minipage}{0.30\textwidth}
		\centering
		\includegraphics[height=3.75cm,width=4.75cm]{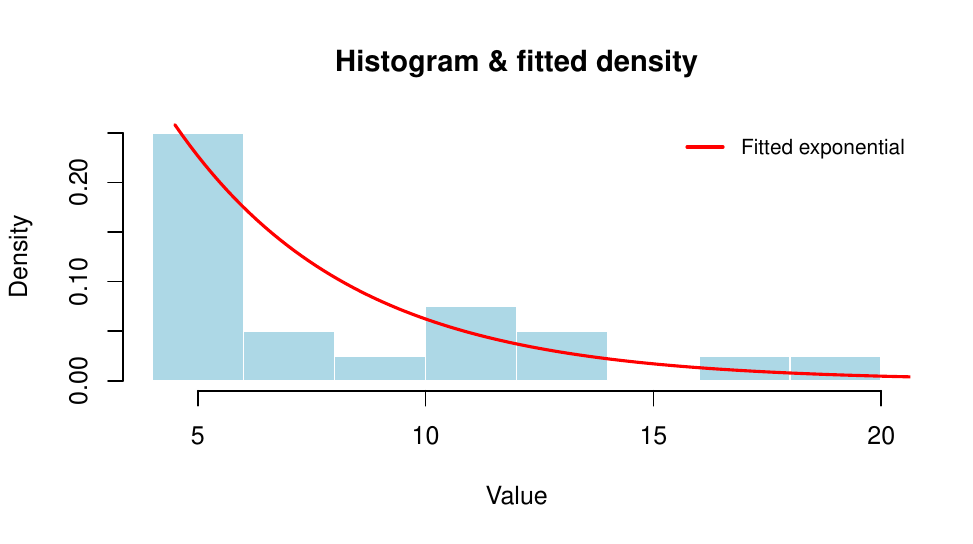}
	\end{minipage}
	\hspace{0.02\textwidth}
	\begin{minipage}{0.30\textwidth}
		\centering
		\includegraphics[height=3.75cm,width=4.75cm]{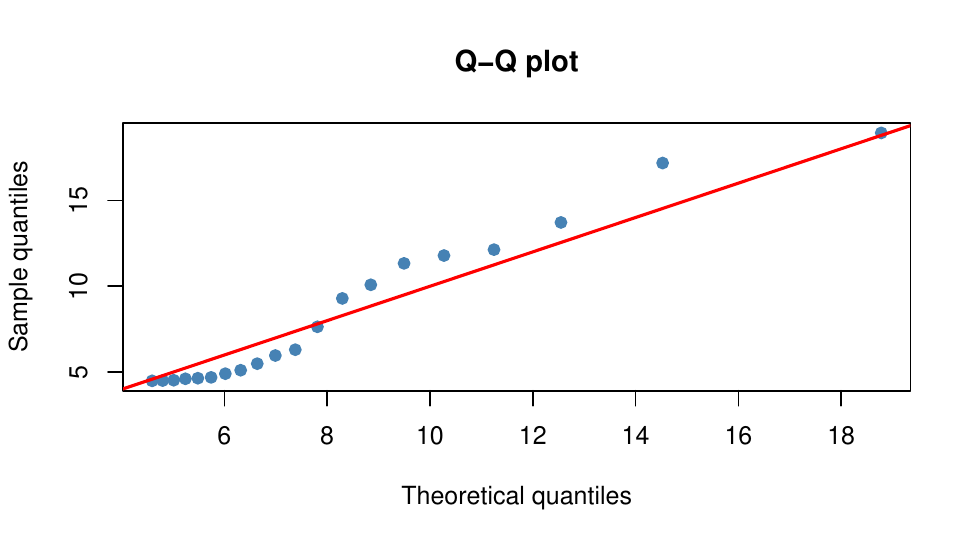}
	\end{minipage}
	\hspace{0.02\textwidth}
	\begin{minipage}{0.30\textwidth}
		\centering
		\includegraphics[height=3.75cm,width=4.75cm]{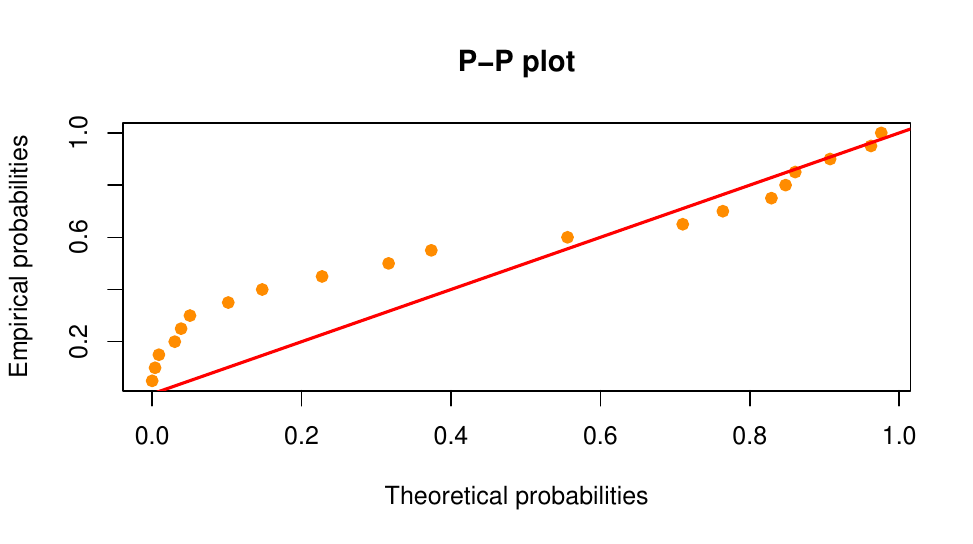}
	\end{minipage}
	\vspace{0.00cm}
	\vspace{0.00cm}
	\begin{minipage}{0.30\textwidth}
		\centering
		\includegraphics[height=3.75cm,width=4.75cm]{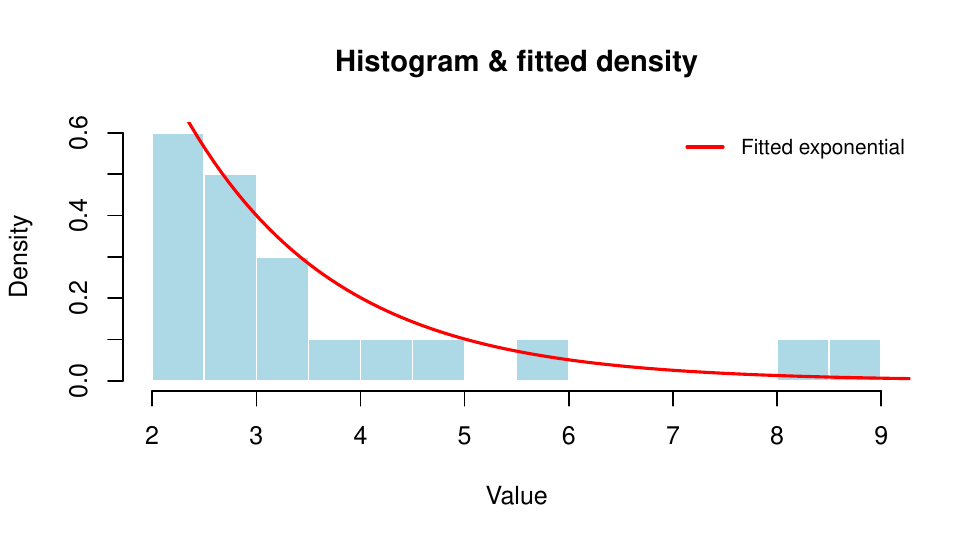}
	\end{minipage}
	\hspace{0.02\textwidth}
	\begin{minipage}{0.30\textwidth}
		\centering
		\includegraphics[height=3.75cm,width=4.75cm]{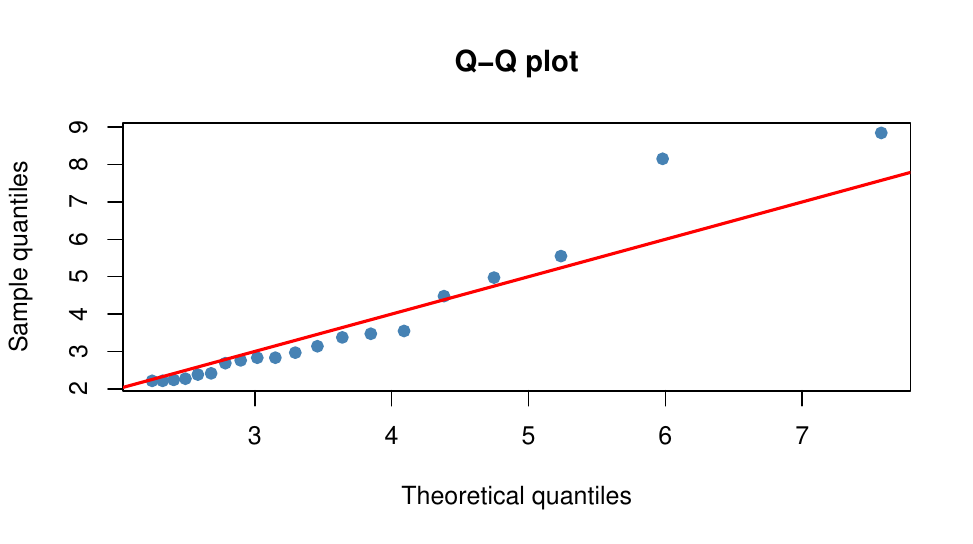}
	\end{minipage}
	\hspace{0.02\textwidth}
	\begin{minipage}{0.30\textwidth}
		\centering
		\includegraphics[height=3.75cm,width=4.75cm]{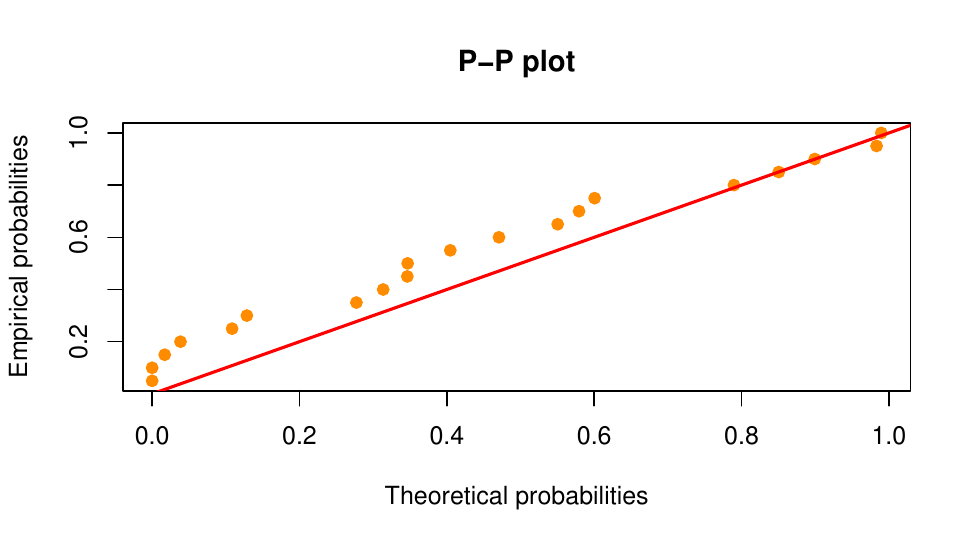}
	\end{minipage}
	\vspace{0.05cm}
	\vspace{0.00cm}
	\begin{minipage}{0.30\textwidth}
		\centering
		\includegraphics[height=3.75cm,width=4.75cm]{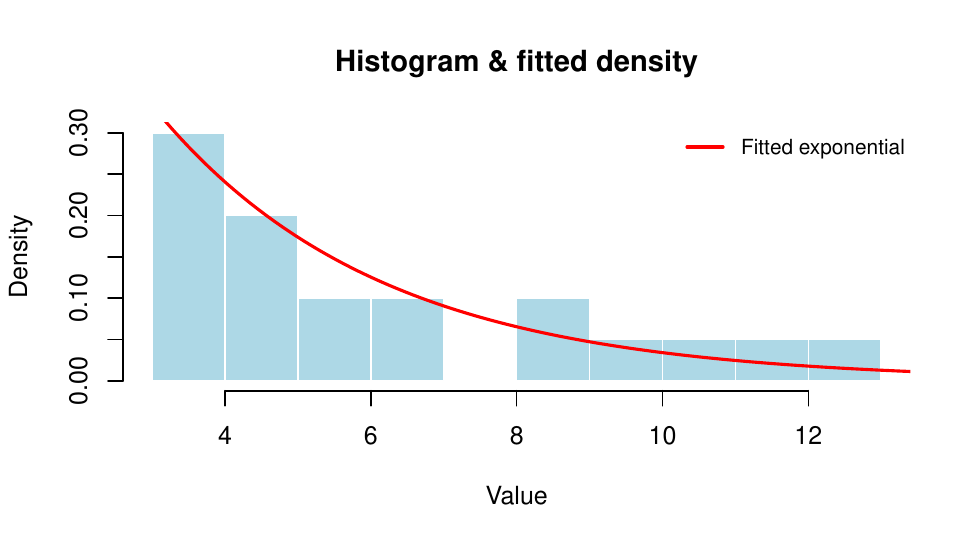}
	\end{minipage}
	\hspace{0.02\textwidth}
	\begin{minipage}{0.30\textwidth}
		\centering
		\includegraphics[height=3.75cm,width=4.75cm]{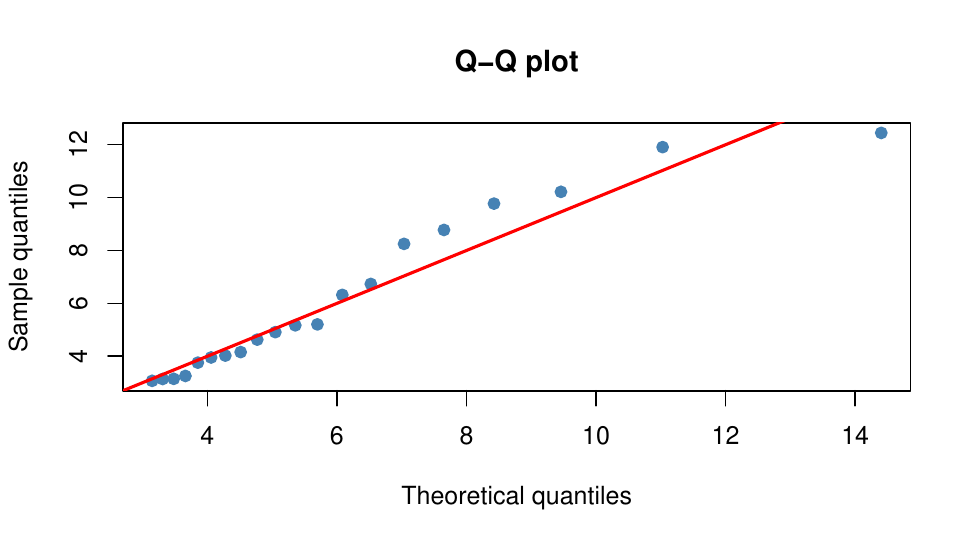}
	\end{minipage}
	\hspace{0.02\textwidth}
	\begin{minipage}{0.30\textwidth}
		\centering
		\includegraphics[height=3.75cm,width=4.75cm]{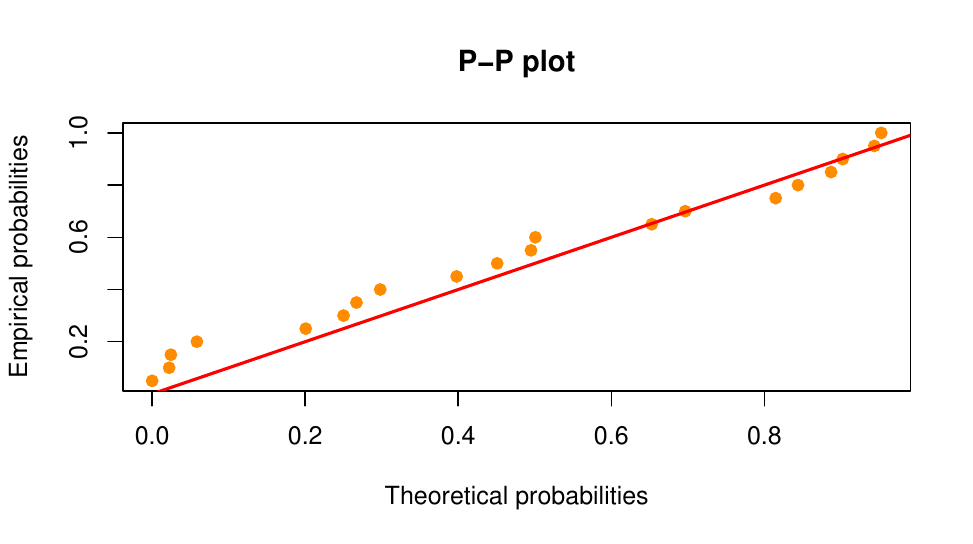}
	\end{minipage}
	\begin{minipage}{0.70\textwidth}
		\centering
		\includegraphics[height=8cm,width=15cm]{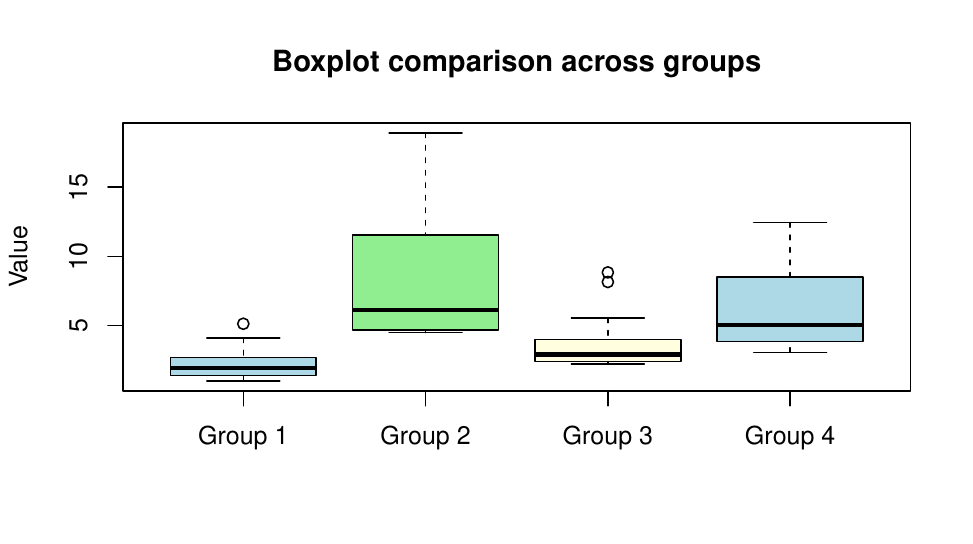}
	\end{minipage}
	
	\caption{Graphical representations of the four drug datasets (drug 1--4) and the corresponding boxplot.}
	
	\label{}
\end{figure}

\begin{table}[h]
	\centering
	\caption{Parametric bootstrap simultaneous confidence intervals for differences of means under 10\% doubly censored data}
	\label{}
	\small
	\begin{tabular}{l ccc ccc}
		\toprule
		& \multicolumn{3}{c}{\textbf{Unbiased (Sum)}} & \multicolumn{3}{c}{\textbf{MLE (Sum)}} \\
		\cmidrule(lr){2-4} \cmidrule(lr){5-7}
		\textbf{Parameter} & Lower & Upper & Length & Lower & Upper & Length \\
		\midrule
		$\theta_1-\theta_2$ & -8.7486 & -3.6124 & 5.1363 & -8.7364 & -3.6246 & 5.1119 \\
		$\theta_1-\theta_3$ & -2.6658 & -0.2920 & 2.3737 & -2.6601 & -0.2977 & 2.3624 \\
		$\theta_1-\theta_4$ & -6.0416 & -1.8661 & 4.1754 & -6.0316 & -1.8761 & 4.1556 \\
		$\theta_2-\theta_3$ &  2.0768 &  7.3264 & 5.2497 &  2.0892 &  7.3140 & 5.2247 \\
		$\theta_2-\theta_4$ & -0.9102 &  5.3635 & 6.2736 & -0.8953 &  5.3486 & 6.2439 \\
		$\theta_3-\theta_4$ & -4.6320 & -0.3179 & 4.3141 & -4.6218 & -0.3281 & 4.2937 \\
		\bottomrule
	\end{tabular}
	
	\vspace{0.5cm}
	
	\begin{tabular}{l ccc ccc}
		\toprule
		& \multicolumn{3}{c}{\textbf{Unbiased (Max)}} & \multicolumn{3}{c}{\textbf{MLE (Max)}} \\
		\cmidrule(lr){2-4} \cmidrule(lr){5-7}
		\textbf{Parameter} & Lower & Upper & Length & Lower & Upper & Length \\
		\midrule
		$\theta_1-\theta_2$ &-9.2057 & -3.1553 & 6.0503 &-9.1803 & -3.1807 & 5.9996 \\
		$\theta_1-\theta_3$ &-2.6149 & -0.3429 & 2.2721 &-2.6054 & -0.3524 & 2.2530 \\
		$\theta_1-\theta_4$ &-6.3542 & -1.5535 & 4.8006 &-6.3340 & -1.5737 & 4.7604 \\
		$\theta_2-\theta_3$ & 1.6764 &  7.7268 & 6.0503 & 1.7018 &  7.7014 & 5.9996 \\
		$\theta_2-\theta_4$ &-0.7985 &  5.2518 & 6.0503 &-0.7731 &  5.2264 & 5.9996 \\
		$\theta_3-\theta_4$ &-4.8753 & -0.0746 & 4.8006 &-4.8551 & -0.0948 & 4.7604 \\
		\bottomrule
	\end{tabular}
\end{table}

\begin{table}[h]
	\centering
	\caption{Fiducial generalized  simultaneous confidence intervals for differences of means under 10\% Doubly censored data}
	\label{}
	\small
	\begin{tabular}{l ccc ccc}
		\toprule
		& \multicolumn{3}{c}{\textbf{Unbiased (Sum)}} & \multicolumn{3}{c}{\textbf{MLE (Sum)}} \\
		\cmidrule(lr){2-4} \cmidrule(lr){5-7}
		\textbf{Parameter} & Lower & Upper & Length & Lower & Upper & Length \\
		\midrule
		$\theta_1-\theta_2$ &-9.1633 &-3.1977 &5.9655 &-9.0869 &-3.2741 &5.8128 \\
		$\theta_1-\theta_3$ &-2.8574 &-0.1004 &2.7570 &-2.8221 &-0.1357 &2.6864 \\
		$\theta_1-\theta_4$ &-6.3786 &-1.5291 &4.8496 &-6.3165 &-1.5912 &4.7254 \\
		$\theta_2-\theta_3$ & 1.6530 & 7.7502 &6.0972 & 1.7311 & 7.6721 &5.9411 \\
		$\theta_2-\theta_4$ &-1.4166 & 5.8699 &7.2865 &-1.3233 & 5.7766 &7.0999 \\
		$\theta_3-\theta_4$ &-4.9803 & 0.0304 &5.0107 &-4.9161 &-0.0338 &4.8823 \\
		\bottomrule
	\end{tabular}
	
	\vspace{0.5cm}
	
	\begin{tabular}{l ccc ccc}
		\toprule
		& \multicolumn{3}{c}{\textbf{Unbiased (Max)}} & \multicolumn{3}{c}{\textbf{MLE (Max)}} \\
		\cmidrule(lr){2-4} \cmidrule(lr){5-7}
		\textbf{Parameter} & Lower & Upper & Length & Lower & Upper & Length \\
		\midrule
		$\theta_1-\theta_2$ &-9.5262 &-2.8348 &6.6914 & -9.4245 &-2.9365 &6.4880\\
		$\theta_1-\theta_3$ &-2.7353 &-0.2225 &2.5128 & -2.6971 &-0.2607 &2.4365\\
		$\theta_1-\theta_4$ &-6.6085 &-1.2992 &5.3093 & -6.5278 &-1.3799 &5.1479\\
		$\theta_2-\theta_3$ & 1.3559 & 8.0473 &6.6914 &  1.4576 & 7.9456 &6.4880\\
		$\theta_2-\theta_4$ &-1.1191 & 5.5724 &6.6914 & -1.0174 & 5.4707 &6.4880\\
		$\theta_3-\theta_4$ &-5.1296 & 0.1797 &5.3093 & -5.0489 & 0.0990 &5.1479\\
		\bottomrule
	\end{tabular}
\end{table}

\clearpage
\appendix
\section{Appendix}\label{VII}
\begin{theorem}[\cite{gibbons2020nonparametric}]\label{A1}
	Let $X_{(r)}$ denote the $r$-th order statistic of a random sample of size $n$ from any continuous cumulative distribution function $F_{X}(x)$. Then if $\frac{r}{n} \to p$ as $n \to \infty,~0<p<1$, the distribution of 
	$$\left[\frac{n}{p(1-p)}\right]^{\frac{1}{2}} f_X(\theta)\left[X_(r)-\theta\right]\to \mathcal{N}(0,1), ~~ \theta = F_X^{-1}(p).$$
\end{theorem}	
\begin{lemma}\label{A2}
	Let $r_{i}-1=n_{i}p_{i}$. Then, as $n_i \to \infty$,
	\begin{itemize}
		\item [(i)] $\sum\limits_{j=n_{i}-r_{i}+1}^{n_{i}}\frac{1}{j} \to -\ln(1-p_{i})$ 
		\item [(ii)] $\sqrt{n_{i}}\left(\ln(1-p_{i})+\sum\limits_{j=n_{i}-r_{i}+1}^{n_{i}}\frac{1}{j}\right) \to 0.$
	\end{itemize}
\end{lemma}
\textbf{Proof} (i)  We have $$\sum\limits_{j=n_{i}-r_{i}+1}^{n_{i}}\frac{1}{j}=\sum\limits_{j=1}^{n_{i}}\frac{1}{j}-\sum\limits_{j=1}^{n_{i}-r_{i}+1}\frac{1}{j}=H_{n_{i}}-H_{n_{i}r_{i}}$$ 
Where $H_n=\sum\limits_{j=1}^{n}\frac{1}{j}$. Using the asymptotic expansion $H_n=\ln n +\nu+O(\frac{1}{n}),$ where $\nu$ is Euler's constant (see \cite{graham1994concrete}) and $O(\frac{1}{n_{i}})$ describes a quantity that shrinks towards zero as $n_i$ grows larger. We obtain
$$H_{n_{i}}-H_{n_{i}-r_{i}}=\ln n_i+\nu+O\left(\frac{1}{n_{i}}\right)-\ln(n_i-r_i)-\nu-O\left(\frac{1}{n_{i}-r_{i}}\right)=\ln\left(\frac{n_i}{n_i-r_i}\right)+O\left(\frac{1}{n_{i}}\right).$$
Since $\frac{n_{i}}{n_{i}-r_{i}}=\frac{1}{1-\frac{r_{i}}{n_{i}}}$, this gives 
$$\sum\limits_{j=n_{i}-r_{i}+1}^{n_{i}}\frac{1}{j}=\ln\left(\frac{1}{1-\frac{r_i}{n_i}}\right)+O\left(\frac{1}{n_{i}}\right).$$
Finally, the condition $r_{i}-1=n_{i}p_{i}$ implies $\frac{r_{i}}{n_{i}} \to p_{i}$ as $n_{i} \to \infty$ and $O(\frac{1}{n_{i}})\to 0$. Hence
$$\sum\limits_{j=n_{i}-r_{i}+1}^{n_{i}}\frac{1}{j}\to \ln\left(\frac{1}{1-p_{i}}\right) = -\ln(1-p_{i})$$

\noindent (ii) From the computation in part (i), we know precisely that 
$$\sum\limits_{j=n_{i}-r_{i}+1}^{n_{i}}\frac{1}{j}=-\ln{(1-p_{i})}+O\left(\frac{1}{n_{i}}\right).$$
Substituting this into the expression in part II, the leading term cancel
$$\ln(1-p_{i})+\sum\limits_{j=n_{i}-r_{i}+1}^{n_{i}}\frac{1}{j}=\ln(1-p_{i})-\ln(1-p_{i})+O\left(\frac{1}{n_{i}}\right)=O\left(\frac{1}{n_{i}}\right).$$ 
Multiplying by $\sqrt{n}$ gives 
$$\sqrt{n_{i}}\left(\ln(1-p_{i})+\sum\limits_{j=n_{i}-r_{i}+1}^{n_{i}}\frac{1}{j}\right)=O\left(\frac{1}{\sqrt{n_{i}}}\right).$$
Since $O(\frac{1}{\sqrt{n_i}})\to 0$ as $n_{i} \to \infty$, we conclude that 
$$\sqrt{n_{i}}\left(\ln(1-p_{i})+\sum\limits_{j=n_{i}-r_{i}+1}^{n_{i}}\frac{1}{j}\right)\to 0.$$

\end{document}